\documentclass{amsart}

\usepackage{amssymb}
\usepackage{amsmath}
\usepackage{stmaryrd}
\usepackage{xparse}
\usepackage[breaklinks,colorlinks=true,allcolors=blue]{hyperref}
\usepackage{tikz}
\usepackage{todonotes}

\theoremstyle{plain}
\newtheorem{theorem}{Theorem}
\newtheorem{corollary}[theorem]{Corollary}
\newtheorem{lemma}[theorem]{Lemma}
\newtheorem{proposition}[theorem]{Proposition}

\theoremstyle{remark}
\newtheorem{remark}[theorem]{Remark}

\numberwithin{equation}{section}
\numberwithin{theorem}{section}

\newcommand{\R}{\mathbb{R}}
\newcommand{\Domain}{\Omega}
\newcommand{\Dim}{d}

\newcommand{\Leb}[1]{L^2(#1)}
\newcommand{\LebH}[1]{L^2_0(#1)}

\newcommand{\SobH}[1]{H^1_0(#1)}
\newcommand{\SobHD}[1]{H^{-1}(#1)}
\newcommand{\Norm}[1]{\| #1 \|}
\newcommand{\NormLeb}[2]{\Norm{#1}_{#2}}
\newcommand{\Normtr}[1]{ \lvert\!\lvert\!\lvert{#1} \rvert\!\rvert\!\rvert}  

\newcommand{\Grad}{\nabla}
\newcommand{\SymGrad}{\varepsilon}
\newcommand{\Div}{\mathrm{div}}
\newcommand{\GradM}{\Grad_\Mesh}
\newcommand{\SymGradM}{\SymGrad_\Mesh}
\newcommand{\DivM}{\Div_\Mesh}

\newcommand{\Smt}{\mathcal{E}}
\newcommand{\Rinv}{\mathcal{R}}
\newcommand{\Riesz}{\mathcal{L}}
\newcommand{\LebPro}[1]{\mathcal{P}_{#1}}
\newcommand{\Interp}{\mathcal{I}}

\newcommand{\Mesh}{\mathfrak{T}}
\newcommand{\Faces}{\mathfrak{F}}
\newcommand{\FacesInt}{{\Faces^i}}
\newcommand{\Verts}{\mathfrak{V}}
\newcommand{\VertsInt}{\Verts^i}
\newcommand{\Skel}{{\partial \Mesh}}
\newcommand{\MeshSize}{\mathsf{h}}
\newcommand{\MeshEl}{\mathsf{T}}
\newcommand{\FaceEl}{\mathsf{F}}
\newcommand{\VertsEl}{\mathsf{z}}
\newcommand{\Normal}{\mathsf{n}}
\newcommand{\Jump}[1]{\llbracket #1 \rrbracket}
\newcommand{\Avg}[1]{\{\!\!\{ #1\}\!\!\}}
\newcommand{\Shape}[1]{\gamma(#1)}

\newcommand{\Degree}{\ell}
\newcommand{\Poly}[1]{\mathbb{P}_{#1}}
\newcommand{\PolyPiec}[1]{\mathbb{S}_{#1}}
\newcommand{\PolyPiecAvg}[1]{\mathbb{S}_{#1,0}}
\newcommand{\CR}{\mathbb{CR}}
\newcommand{\CRLabel}{\mathrm{CR}}
\newcommand{\DGLabel}{\mathrm{dG}}
\newcommand{\Err}[1]{\mathrm{ERR}(#1)}
\newcommand{\ERR}{\mathrm{ERR}}

\begin{document}

\title[A nonsymmetric approach to the Biot's model I]{A nonsymmetric approach and a quasi-optimal and robust discretization for the Biot's model. \\Part I -- Theoretical aspects}

\author[A.~Khan]{Arbaz Khan}
\address{Department of Mathematics,	Indian Institute of Technology Roorkee (IITR), 247667 Roorkee, India}
\email{arbaz@ma.iitr.ac.in}

\author[P.~Zanotti]{Pietro Zanotti}
\address{Dipartimento di Matematica `F. Enriques', Universit\`{a} degli Studi di Milano, 20133 Milano, Italy}
\email{pietro.zanotti@unimi.it}



\begin{abstract}
We consider the system of partial differential equations stemming from the time discretization of the two-field formulation of the Biot's model with the backward Euler scheme. A typical difficulty encountered in the space discretization of this problem is the robustness with respect to various material parameters. We deal with this issue by observing that the problem is uniformly stable, irrespective of all parameters, in a suitable nonsymmetric variational setting. Guided by this result, we design a novel nonconforming discretization, which employs Crouzeix-Raviart and discontinuous elements. We prove that the proposed discretization is quasi-optimal and robust in a parameter-dependent norm and discuss the consequences of this result.   	
\end{abstract} 

\maketitle

\section{Introduction}
\label{S:introduction}

In the theory of poroelasticity, the Biot's consolidation model describes the flow of a fluid inside an elastic porous medium. The model has a wide range of applications covering, e.g., reservoir engineering, biomechanics and medicine. For this reason, the problem of devising effective discretization techniques has attracted increasing attention in recent years. 

In the basic two-field formulation of the Biot's model, the unknowns are the displacement of the medium and the fluid pressure. Various other formulations have been introduced over the years, in order to devise discretizations enjoying desirable properties, such as the robustness with respect to some of the material parameters and/or the conservation of relevant quantities. Three-field formulations introducing the Darcy's velocity are well-established, see e.g. \cite{Phillips.Wheeler:07,Phillips.Wheeler:08,Yi.13,Berger.Bordas.Kay.Tavener:15,Hu.Rodrigo.Gaspar.Zikatanov:17,Hong.Kraus:18}. More recently, a three-field formulation involving the so-called total pressure \cite{Oyarzua.RuizBaier:16,Lee.Mardal.Winther:17} and various other four-field formulations \cite{Korsawe.Starke:05,Haga.Osnes.Langtangen:12,Yi.14,Lee:16,Kumar.Oyarzua.RuizBaier.Sandilya:20} have been considered.

The Biot's model involves several material parameters. For extreme values of certain parameters, undesired numerical effects are possibly met in the discretization. In particular, volumetric locking and spurious oscillations of the fluid pressure may be observed when dealing with nearly incompressible and low permeable materials, respectively. The different nature of these effects is pointed out in \cite{Haga.Osnes.Langtangen:12}. The devising of discretizations that are robust in the critical regimes has been the subject of many papers. In addition to the aforementioned ones, we refer also to \cite{Chen.Luo.Feng:13,Nordbotten:16,Rodrigo.Gaspar.Hu.Zikatanov:16}. The robustness is typically achieved by employing discrete spaces that satisfy certain inf-sup conditions, namely the ones encountered in the approximation of the mixed formulation of the Stokes and of the Poisson problems see, for instance, \cite[section~4]{Haga.Osnes.Langtangen:12}, \cite[Definition~3.1]{Rodrigo.Hu.Ohm.Adler.Gaspar.Zikatanov:18} and \cite{Mardal.Rognes.Thompson:20}.

The loss of mass is another issue that possibly affects the discretization of the Biot's model. Treating the Darcy's velocity as an independent variable is a standard technique to assess local mass conservation. Among the other strategies we mention, for instance, the flux reconstruction of \cite{Riedlbeck.DiPietro.Ern.Granet.Kazymyrenko:17}.

In this paper we first propose a novel nonsymmetric variational setting for the two-field formulation of the problem resulting from the time discretization of the Biot's model with the backward Euler scheme. We equip the trial space with a parameter-dependent norm and the test space with a norm that is parameter-independent, after rescaling the equations. The motivation behind our approach is that we do not assume any scaling of the load terms with respect to the parameters. We prove well-posedness and stability in the proposed setting, irrespective of all parameters. The expression of the trial norm points out an equivalent nonsymmetric four-field formulation, treating the total pressure and the total fluid content as independent variables.

The proof of our stability result builds upon the continuity and the inf-sup stability of two auxiliary bilinear forms: the well-known $H^1$/$L^2$ form involved in the Stokes problem and the $H^{-1}$/$H^1_0$ dual pairing. The first one serves to control the $L^2$-norm of the total pressure, as done in \cite{Lee.Mardal.Winther:17}. The second one, which appears to be a new device in this context, is invoked in order to control the $H^{-1}$-norm of the total fluid content. 

The second contribution of this paper is a finite element discretization inspired by the aforementioned result. We use first-order Crouzeix-Raviart and discontinuous elements, respectively, for the displacement and for the fluid pressure. Our discretization equivalently reads as an approximation of the nonsymmetric four-field formulation of the problem, employing piecewise constants for both the total pressure and the total fluid content. The pairs of spaces used for the displacement and the total pressure and for the total fluid content and the fluid pressure are suitable for establishing counterparts of the continuity and of the inf-sup inequalities exploited in the analysis of the model problem. This yields a robust stability estimate and prevents, in particular, from volumetric locking and substantial mass losses. In contrast, the pair of spaces used for the displacement and the fluid pressure is not suitable for a Stokes-like inf-sup inequality, a property that is known to prevent from spurious pressure oscillations in low permeable materials, cf. \cite{Haga.Osnes.Langtangen:12}. Still, we can guarantee that the projection of the fluid pressure onto the piecewise constant functions is free from oscillations, see Remark~\ref{R:spurious-pressure-modes}.

We combine our stability results with a careful discretization of the load terms, inspired by the theory developed in \cite{Veeser.Zanotti:18,Veeser.Zanotti:19b, Veeser.Zanotti:18b}. As a result, we derive quasi-optimal and robust error estimates, meaning that the error of the overall discretization is bounded, up to a constant, by the corresponding best error in the employed spaces. The constant involved in this result is independent of all the material parameters and we do not invoke additional regularity of the solution beyond the minimal one.

The remaining part of the paper is organized as follows. Sections~\ref{S:continuous-problem} and \ref{S:discrete-problem} are devoted to the analysis of the model problem and of its discretization, respectively. In section~\ref{S:extensions} we discuss some extensions of our findings, including higher-order discretizations. Efficient solvers and numerical experiments are discussed in \cite{Khan.Zanotti}. 

\section{A nonsymmetric approach to the Biot's model}
\label{S:continuous-problem}

Let $\Domain \subseteq \R^\Dim$, $\Dim \in \{2, 3\}$, be an open and bounded polyhedron. Assume also that the boundary of $\Domain$ locally coincides with the graph of a Lipschitz-continuous function. For all measurable subsets $\widetilde \Domain \subseteq \Domain$, we denote by $\NormLeb{\cdot}{\widetilde{\Domain}}$ the $L^2$-norm on $\widetilde \Domain$. In this section we recall the equations of the (semi-discrete) Biot's model in $\Domain$ and state a result concerning their stability.

\subsection{Biot's model}
\label{SS:Biot}

The Biot's consolidation model consists of an equilibrium equation, prescribing the conservation of the momentum, and of a continuity equation, stating the conservation of the mass. The two equations read as follows
\begin{equation}
\label{Biot-time-dependent}
\begin{alignedat}{1}
-\Div( 2 \mu \SymGrad(u) + \lambda \Div(u) I - \alpha p_F I  ) 
& = f \quad \text{in} \: \Domain\\
\dfrac{\partial}{\partial t}( \alpha \Div(u) + \sigma p_F) - \Div(\overline{\kappa} \Grad p_F) 
& = \overline{g} \quad \text{in} \: \Domain. 
\end{alignedat}
\end{equation}
The unknowns are the displacement $u: \Domain \to \R^\Dim$ of the elastic medium and the fluid pressure $p_F: \Domain \to \R$. The symbol $\SymGrad(u) := (\Grad u + (\Grad u)^T)/2$ denotes the symmetric gradient of $u$ (i.e. the strain tensor), whereas $I$ stands for the $\Dim \times \Dim$ identity matrix. The model involves the following material parameters
\begin{itemize}
	\item $\lambda, \mu > 0$, the Lam\'{e} coefficients,
	\item $\alpha > 0$, the Biot-Willis constant,
	\item $\sigma \geq 0$, the constrained specific storage coefficient, and
	\item $\overline{\kappa} > 0$, the hydraulic conductivity. 
\end{itemize}
For simplicity, we assume hereafter that all the parameters are constant in $\Domain$. We complement the model by assuming that
\begin{equation}
\label{Biot-BCs}
u = 0 \qquad \text{and} \qquad p_F = 0 \quad \text{on} \: \partial \Domain.
\end{equation}
We refer to section~\ref{SS:mixed-BCs} for a discussion about more general boundary conditions.

The semi-discretization in time of \eqref{Biot-time-dependent} by the backward Euler scheme, with time step $\tau > 0$, results in a time-independent problem in the form
\begin{equation}
\label{Biot-semidiscrete}
\begin{alignedat}{1}
-\Div( 2 \mu \SymGrad(u) + \lambda \Div(u) I - \alpha p_F I  ) 
& = f \quad \text{in} \: \Domain\\
\alpha \Div(u) + \sigma p_F - \Div(\kappa \Grad p_F) 
& = g \quad \text{in} \: \Domain
\end{alignedat}
\end{equation}
where we have, in particular,
\begin{equation*}
\kappa := \tau \overline{\kappa}. 
\end{equation*}
The main concern of this paper is in establishing a stability estimate for this problem and in devising a discretization that are robust with respect to all parameters.

\subsection{Nonsymmetric variational setting}
\label{SS:nonsymmetric-setting}

To obtain a weak formulation of problem~\eqref{Biot-semidiscrete}, we multiply the two equations by smooth test functions and we integrate by parts, as usual. Correspondingly, we assume that the load terms $f$ and $g$ satisfy the regularity requirements
\begin{equation}
\label{loads-regularity}
f \in \SobHD{\Domain; \Dim} := (\SobH{\Domain}^\Dim)^\prime
\qquad \text{and} \qquad 
g \in \SobHD{\Domain} := (\SobH{\Domain})^\prime.
\end{equation} 
Slightly abusing the notation, we denote by $\left\langle \cdot, \cdot \right\rangle $ both the dual pairing of $\SobHD{\Domain; \Dim}$ with $\SobH{\Domain}^\Dim$ and the one of $\SobHD{\Domain}$ with $\SobH{\Domain}$. Then, we are led to the following linear variational problem: 
\begin{equation}
\label{Biot-weak}
\begin{gathered}[c]
\text{find} \quad (u, p_F)  \in \SobH{\Domain}^\Dim \times \SobH{\Domain} \quad \text{such that}\\
\forall (v, q_F) \in \SobH{\Domain}^\Dim \times \SobH{\Domain} 
\qquad 
b((u, p_F), (v, q_F))
=
\left\langle f, v\right\rangle + \left\langle g, q_F\right\rangle 
\end{gathered}
\end{equation}
where the bilinear form $b$ is given by
\begin{equation}
\label{form-b}
\begin{alignedat}{1}
b( (\widetilde u, \widetilde{p}_F), (v, q_F) ) := \:
& 2 \mu \int_\Domain \SymGrad(\widetilde u) \colon \SymGrad( v) 
+ \int_\Domain (\lambda \Div(\widetilde u) - \alpha \widetilde{p}_F) \Div(v)\\
& + \int_\Domain ( \alpha \Div(\widetilde{u}) + \sigma \widetilde{p}_F )q_F + \kappa \int_\Domain \Grad \widetilde{p}_F \cdot \Grad q_F
\end{alignedat}
\end{equation}
for all $(\widetilde u, \widetilde p_F), (v, q_F) \in \SobH{\Domain}^\Dim \times \SobH{\Domain}$. 

The existence and the uniqueness of the solution of this problem are well established, cf. Corollary~\ref{C:stability} below. Still, the devising and the analysis of robust discretizations require also sharp results concerning the stability of the solution. For this purpose, one usually introduces a norm $\Norm{\cdot}_1$ on the trial space (i.e. the space of all possible solutions) and a norm $\Norm{\cdot}_2$ on the test space (i.e. the space of all possible test functions). Then, assuming that the form $b$ is continuous and inf-sup stable in these norms, the equivalence
\begin{equation}
\label{stability-abstract}
\Norm{(u, p_F)}_1 \approx \Norm{(f, g)}_{2, \star}
\end{equation}   
is readily derived, where $\Norm{\cdot}_{2, \star}$ is the norm dual to $\Norm{\cdot}_2$ with respect to the dual pairing $\left\langle \cdot, \cdot \right\rangle $.

In our case, the trial and the test spaces coincide, thus suggesting to equip them by the same norm $\Norm{\cdot}$. (Note also that $b$ is symmetric, up to replacing the test function $q_F$ by $-q_F$.) This approach is quite popular, but it has the disadvantage that $\Norm{\cdot}$ must be parameter-dependent, in order to make the constants hidden in \eqref{stability-abstract} parameter-independent. Therefore, the norm dual to $\Norm{\cdot}$ must be parameter-dependent as well and we cannot ensure that the solution of problem~\eqref{Biot-weak} is uniformly bounded in $\Norm{\cdot}$, irrespective of all parameters, only by assumption \eqref{loads-regularity}.  

When additional informations on the load terms beyond \eqref{loads-regularity} are not available, it seems advisable to analyze the model problem in a nonsymmetric setting, by equipping the test space with a parameter-independent norm $\Norm{\cdot}_2$ and then looking for a corresponding trial norm $\Norm{\cdot}_1$ such that the continuity and the inf-sup constants of $b$ are parameter-independent. We deviate from this principle only in that we rescale the first equation of \eqref{Biot-semidiscrete} by $\sqrt{\kappa}$ and the second one by $\sqrt{2 \mu}$, to make sure that the scales in the two equations are balanced. In other words, we consider the test norm
\begin{equation}
\label{norm-test}
\Norm{(v, q_F)}_2 := 
\left( \dfrac{\NormLeb{\SymGrad(v)}{\Domain}^2}{\kappa} + 
\dfrac{\NormLeb{\Grad q_F}{\Domain}^2}{2 \mu} \right)^{\frac{1}{2}},   
\end{equation}
for $(v, q_F) \in \SobH{\Domain}^\Dim \times \SobH{\Domain}$. The importance of such rescaling is made clear by the proof of Theorem~\ref{T:cont-infsup} below, cf. Remark~\ref{R:alternative-test-norm}. The corresponding dual norm is 
\begin{equation}
\label{norm-test-dual}
\begin{split}
\Norm{(f, g)}_{2, \star} &:= 
\sup_{(v, q_F) \in \SobH{\Domain}^\Dim \times \SobH{\Domain}}
\dfrac{\left\langle f, v\right\rangle + \left\langle g, q_F\right\rangle }{\Norm{(v, q_F)}_2}\\
&=
\left( \kappa \Norm{f}_{\SobHD{\Domain; \Dim}}^2
+ 2 \mu \Norm{g}_{\SobHD{\Domain}}^2 \right)^{\frac{1}{2}}
\end{split}
\end{equation}
where $(f, g) \in \SobHD{\Domain; \Dim} \times \SobHD{\Domain}$.

Thus, we aim at finding a corresponding trial norm $\Norm{\cdot}_1$ so that the bilinear form $b$ is uniformly continuous and inf-sup stable, irrespective of all parameters. For this purpose, we denote by $\LebPro{\mathbb{W}}: \Leb{\Domain} \to \mathbb{W}$ the $L^2$-orthogonal projection onto a closed subspace $\mathbb{W} \subseteq \Leb{\Domain}$. Such projection is determined through the problem
\begin{equation*}
\forall w\in \mathbb{W} \qquad 
\int_\Domain \LebPro{\mathbb{W}}(q) w 
= 
\int_\Domain qw
\end{equation*}
for all $q \in \Leb{\Domain}$. 

The subspace of all $L^2$-functions with vanishing average over $\Domain$ is
\begin{equation*}
\label{lebH}
\LebH{\Domain} :=
\{ q_0 \in \Leb{\Domain} \mid \textstyle \int_\Domain q_0 = 0 \}.
\end{equation*}
The divergence operator maps $\SobH{\Domain}^\Dim$ onto $\LebH{\Domain}$ and, for all $q_0 \in \LebH{\Domain}$, we have the equivalence
\begin{subequations}
\label{cont-infsup-auxiliary}
\begin{equation}
\label{cont-infsup-stokes}
\sup_{v \in \SobH{\Domain}^\Dim} 
\dfrac{\int_\Domain q_0 \Div(v)}{\NormLeb{\SymGrad(v)}{\Domain}}
\approx
\NormLeb{q_0}{\Domain}
\end{equation}
where the hidden constants only depend on $\Domain$. This result follows from the continuity and the inf-sup stability of the $H^1/L^2$ bilinear form involved in the mixed formulation of the Stokes equations, combined with the Korn's first inequality, see, e.g., \cite[example~4.2.2]{Boffi.Brezzi.Fortin:13}.

Recall also the embedding $\Leb{\Domain} \hookrightarrow \SobHD{\Domain}$. For all $q \in \Leb{\Domain}$, the identity
\begin{equation}
\label{cont-infsup-riesz}
\sup_{q_F \in \SobH{\Domain}}
\dfrac{\int_\Domain q q_F}{\NormLeb{\Grad q_F}{\Domain}}
=
\Norm{q}_{\SobHD{\Domain}}
\end{equation}  
\end{subequations}
readily follows from the definition of the $H^{-1}$-norm.

We are now in position to state the main result of this section.

\begin{theorem}[Continuity and inf-sup stability of $b$]
\label{T:cont-infsup}
Let the form $b$ be defined as in \eqref{form-b}. For all $(\widetilde{u}, \widetilde{p}_F) \in \SobH{\Domain}^\Dim \times \SobH{\Domain}$, it holds that
\begin{equation}
\label{cont-infsup}
\sup_{(v, q_F) \in \SobH{\Domain}^\Dim \times \SobH{\Domain}}
\dfrac{b( (\widetilde{u}, \widetilde{p}_F), (v, q_F) )}{\Norm{(v, q_F)}_2} 
\approx
\Norm{(\widetilde{u}, \widetilde{p}_F)}_1
\end{equation}
where the hidden constants only depend on $\Domain$, the norm $\Norm{\cdot}_2$ is as in \eqref{norm-test} and
\begin{equation}
\label{norm-trial}
\begin{alignedat}{2}
\Norm{(\widetilde{u}, \widetilde{p}_F)}_1 := 
\Big ( & \mu^2\kappa \NormLeb{\SymGrad(\widetilde u)}{\Domain}^2 
+ \kappa \NormLeb{\lambda \Div(\widetilde u) - \alpha \LebPro{\LebH{\Domain}}(\widetilde p_F)}{\Domain}^2\\
+& \mu \kappa^2 \NormLeb{\Grad \widetilde p_F}{\Domain}^2 
+  \mu \Norm{\alpha \Div(\widetilde u) + \sigma \widetilde p_F}_{\SobHD{\Domain}}^2
\Big )^{\frac{1}{2}}.  
\end{alignedat}
\end{equation} 
\end{theorem}

\begin{proof}
The derivation of the upper bound `$\lesssim$' in \eqref{cont-infsup} is straight-forward, so we only prove the lower bound `$\gtrsim$'. First of all, we rewrite $b$ in a more convenient way, by means of two operators $\Riesz_\SymGrad: \LebH{\Domain} \to \SobH{\Domain}^\Dim$ and $\Riesz_\Grad: \Leb{\Domain} \to \SobH{\Domain}$. For all $q_0 \in \LebH{\Domain}$, we define $\Riesz_\SymGrad (q_0) \in \SobH{\Domain}^\Dim$ through the problem
\begin{equation*}
\label{Rdiv}
\forall v \in \SobH{\Domain}^\Dim
\qquad
\int_\Domain \SymGrad(\Riesz_\SymGrad(q_0)) \colon \SymGrad(v)
=
\int_\Domain q_0 \Div(v).
\end{equation*} 	
The equivalence \eqref{cont-infsup-stokes} entails that we have
\begin{equation}
\label{Rdiv-norm}
\NormLeb{\SymGrad(\Riesz_\SymGrad(q_0))}{\Domain}
=
\sup_{v \in \SobH{\Domain}^\Dim}
\dfrac{\int_\Domain \SymGrad(\Riesz_\SymGrad(q_0)) \colon \SymGrad(v) }{\NormLeb{\SymGrad(v)}{\Domain}}
\approx
\NormLeb{q_0}{\Domain}.
\end{equation}
Similarly, for all $q \in \Leb{\Domain}$, we define $\Riesz_\Grad(q) \in \SobH{\Domain}$ via the problem
\begin{equation*}
\label{Riesz}
\forall q_F \in \SobH{\Domain}
\qquad
\int_\Domain \Grad \Riesz_\Grad(q) \cdot \Grad q_F
= 
\int_\Domain q q_F
\end{equation*}
and it holds that
\begin{equation}
\label{Riesz-norm}
\NormLeb{\Grad \Riesz_\Grad(q)}{\Domain}
=
\sup_{q_F \in \SobH{\Domain}}
\dfrac{\int_\Domain \Grad\Riesz_\Grad(q) \cdot \Grad q_F }{\NormLeb{\Grad q_F}{\Domain}}
=
\NormLeb{q}{\SobHD{\Domain}}
\end{equation}
in view of \eqref{cont-infsup-riesz}. Hence, we see that
\begin{equation*}
\begin{alignedat}{1}
b( (\widetilde u, \widetilde{p}_F), (v, q_F) ) =
&  \int_\Domain \SymGrad( 2 \mu \widetilde u +
\Riesz_\SymGrad(\lambda \Div(\widetilde u) - \alpha \LebPro{\LebH{\Domain}}(\widetilde{p}_F))) \colon  \SymGrad(v)\\
+&  \int_\Domain \Grad( \Riesz_\Grad(\alpha \Div(\widetilde{u}) + \sigma \widetilde{p}_F ) + \kappa \widetilde{p}_F ) \cdot \Grad q_F
\end{alignedat}
\end{equation*}
for all $(\widetilde u, \widetilde p_F), (v, q_F) \in \SobH{\Domain}^\Dim \times \SobH{\Domain}$. This identity and the definition of the test norm $\Norm{\cdot}_2$ reveal that
\begin{equation*}
\label{cont-infsup-proof}
\sup_{(v, q_F) \in \SobH{\Domain}^\Dim \times \SobH{\Domain}}
\dfrac{b( (\widetilde{u}, \widetilde{p}_F), (v, q_F) )}{\Norm{(v, q_F)}_2} = ( \mathfrak{I}_1 + \mathfrak{I}_2  )^{\frac{1}{2}}
\end{equation*}
where
\begin{equation*}
\begin{gathered}
\mathfrak{I}_1 =
\kappa \NormLeb{ \SymGrad( 2\mu \widetilde{u} + \Riesz_\SymGrad(\lambda \Div(\widetilde u) - \alpha \LebPro{\LebH{\Domain}}(\widetilde{p}_F)) )}{\Domain}^2 \\
\mathfrak{I}_2 = 
2 \mu \NormLeb{ \Grad(\Riesz_\Grad(\alpha \Div(\widetilde{u}) + \sigma \widetilde{p}_F) + \kappa \widetilde p_F) }{ \Domain}^2.
\end{gathered}
\end{equation*}
The definition of the operator $\Riesz_\SymGrad$ implies that 
\begin{equation*}
\mathfrak{I}_1 \geq 
4 \mu^2 \kappa \NormLeb{\SymGrad(\widetilde{u})}{\Domain}^2
-
4 \mu \alpha \kappa \int_\Domain \widetilde{p}_F \Div(\widetilde{u})
+
\kappa \NormLeb{\SymGrad(\Riesz_\SymGrad(\lambda \Div(\widetilde u) - \alpha \LebPro{\LebH{\Domain}}(\widetilde{p}_F))) }{\Domain}^2. 
\end{equation*}
Similarly, by recalling the definition of the operator $\Riesz_\Grad$, we infer that
\begin{equation*}
\mathfrak{I}_2 \geq
2 \mu \kappa^2 \NormLeb{\Grad \widetilde p_F}{\Domain}^2
+
4 \mu \alpha \kappa \int_\Domain \widetilde{p}_F \Div(\widetilde{u})
+
2 \mu \NormLeb{\Grad\Riesz_\Grad(\alpha \Div(\widetilde{u}) + \sigma \widetilde{p}_F) }{\Domain}^2.
\end{equation*}
We conclude by inserting these inequalities into the previous identity and by recalling the equivalences \eqref{Rdiv-norm} and \eqref{Riesz-norm}.
\end{proof}

It is worth noticing that the operator $\Riesz_\Grad$ introduced in the proof of Theorem~\ref{T:cont-infsup} is the restriction to $\Leb{\Domain}$ of the Riesz isometry between $\SobHD{\Domain}$ and $\SobH{\Domain}$. Similarly, the operator $\Riesz_\SymGrad$ is the restriction to $\Grad \LebH{\Domain}$ of the Riesz isometry between $\SobHD{\Domain; \Dim}$ and $\SobH{\Domain}^\Dim$, where the gradient is intended in distributional sense. 

\begin{remark}[Equivalent trial norm]
	\label{R:equivalent-trial-norm}
	The proof of Theorem~\ref{T:cont-infsup} reveals that the norm $\Norm{\cdot}_1$ is equivalent to 
	\begin{equation*}
	\label{norm-trial-equivalent}
	\left( \Norm{(\widetilde{u}, \widetilde{p}_F)}^2_1
	+ \mu \lambda \kappa \NormLeb{\Div(\widetilde{u})}{\Domain}^2
	+ \mu \sigma \kappa \NormLeb{p_F}{\Domain}^2 \right)^{\frac{1}{2}}. 
	\end{equation*}
	Indeed, the lower bounds of $\mathfrak{I}_1$ and of $\mathfrak{I}_2$ become two identities when these additional terms are not neglected.
\end{remark}

\begin{remark}[Alternative setting]
\label{R:alternative-test-norm}	
If we drop the scaling factors $\kappa$ and $2\mu$ from the test norm in \eqref{norm-test}, the corresponding trial norm is
\begin{equation*}
\Big ( \; \NormLeb{ \SymGrad( 2\mu \widetilde{u} + \Riesz_\SymGrad(\lambda \Div(\widetilde u) - \alpha \LebPro{\LebH{\Domain}}(\widetilde{p}_F)) )}{\Domain}^2 
+ 
 \NormLeb{ \Grad(\Riesz_\Grad(\alpha \Div(\widetilde{u}) + \sigma \widetilde{p}_F) + \kappa \widetilde p_F) }{ \Domain}^2 
\;\Big )^{\frac{1}{2}}. 
\end{equation*}
By arguing as in the proof of Theorem~\ref{T:cont-infsup}, we see that each one of the above two summands is the sum of three nonnegative terms and of one `mixed' term. The scaling of the test norm considered in \eqref{norm-test} is tailored so as to ensure that the two mixed terms compensate each other.
\end{remark}

The identification of the trial norm $\Norm{\cdot}_1$ in Theorem~\ref{T:cont-infsup} allows us to establish the announced stability estimate for the solution of problem \eqref{Biot-weak}. 

\begin{corollary}[Stability]
\label{C:stability}
For all load terms $(f, g) \in \SobHD{\Domain; \Dim} \times\SobHD{\Domain}$, the problem~\eqref{Biot-weak} is uniquely solvable and its solution $(u, p_F) \in \SobH{\Domain}^\Dim \times \SobH{\Domain}$ fulfills \eqref{stability-abstract},
where the hidden constants only depend on $\Domain$ and the norms $\Norm{\cdot}_1$ and $\Norm{\cdot}_{2, \star}$ are as in \eqref{norm-trial} and \eqref{norm-test-dual}, respectively.
\end{corollary}

\begin{proof}
The existence and the uniqueness of the solution follow from the Banach-Ne\u{c}as theorem, see \cite[Theorem~2.6]{Ern.Guermond:04}. In fact, Theorem~\ref{T:cont-infsup} ensures that the form $b$ is continuous and inf-sup stable. Moreover, we have that
\begin{equation*}
b( (v, q_F), (v, q_F)) = 
2 \mu \NormLeb{\SymGrad(v)}{\Domain}^2
+
\lambda \NormLeb{\Div(v)}{\Domain}^2
+
\sigma \NormLeb{q_F}{\Domain}^2
+
\kappa \NormLeb{\Grad q_F}{\Domain}^2
> 0
\end{equation*}
for all $(v, q_F) \in \SobH{\Domain}^\Dim \times \SobH{\Domain}$, provided $(v, q_F) \neq (0, 0)$. This confirms that all the assumptions in the Banach-Ne\u{c}as theorem are fulfilled. Then, the claimed equivalence \eqref{stability-abstract} readily follows by using \eqref{norm-test-dual} and \eqref{cont-infsup} into problem \eqref{Biot-weak}. 	
\end{proof}

\subsection{A four-fields formulation}
\label{SS:four-fields}

The expression of the trial norm $\Norm{\cdot}_1$, identified in
Theorem~\ref{T:cont-infsup}, suggests that two auxiliary variables are implicitly involved in our analysis. The first one is the so-called total pressure
\begin{equation}
\label{total-pressure}
p_T := \lambda \Div(u) - \alpha \LebPro{\LebH{\Domain}}(p_F), \qquad p_T \in \LebH{\Domain}.
\end{equation} 
Three-fields formulations of the Biot's model treating the total pressure as a third independent unknown, in addition to $u$ and $p_F$, have been recently considered in \cite{Oyarzua.RuizBaier:16,Lee.Mardal.Winther:17}. The second auxiliary variable is the total fluid content
\begin{equation}
\label{total-fluid-content}
m := \alpha \Div(u) + \sigma p_F, 
\qquad m \in \Leb{\Domain}. 
\end{equation} 
Interestingly, the $H^{-1}$-, and not the $L^2$-norm, of $m$ enters into $\Norm{\cdot}_1$. This mismatch is not surprising, when recalling that the problem~\eqref{Biot-weak} originates from the time semi-discretization of \eqref{Biot-time-dependent}.

Though not strictly necessary, it is worth noticing that we obtain a well-posed four-field formulation of the Biot's model by treating both the total pressure and the total fluid content as additional independent unknowns. Indeed, if combine the equations \eqref{total-pressure} and \eqref{total-fluid-content} with problem \eqref{Biot-weak}, then we derive the following nonsymmetric linear variational problem: 
\begin{equation*}
\text{find} \quad (u, p_T, m, p_F) \in 
\SobH{\Domain}^\Dim \times
\LebH{\Domain} \times
\SobHD{\Domain} \times
\SobH{\Domain} \quad \text{such that}
\end{equation*}
\begin{equation}
\label{Biot-four-fields}
\begin{alignedat}{2}
& \forall v \in \SobH{\Domain}^\Dim&
\quad 
2 \mu \int_\Domain\SymGrad(u) \colon \SymGrad(v)
+ \int_\Domain p_T \Div(u)
& = \left\langle f, v\right\rangle\\
& \forall q_T \in \LebH{\Domain}&
\quad
\int_\Domain(\lambda \Div(u) - p_T - \alpha p_F) q_T 
& = 0\\
& \forall s \in \SobH{\Domain}&
\quad
\int_\Domain(\alpha \Div(u) - m + \sigma p_F) s
& = 0\\
& \forall q_F \in \SobH{\Domain}&
\quad
\int_\Domain m q_F + \kappa \int_\Domain \Grad p_F \cdot \Grad q_F
& = \left\langle g, q_F\right\rangle. 
\end{alignedat}
\end{equation}

The approach described in the previous section applies to the analysis of this problem up to minor modifications. In this case, we consider the test norm
\begin{equation*}
\left( 
\dfrac{\NormLeb{\SymGrad(v)}{\Domain}^2}{\kappa} 
+
\dfrac{\NormLeb{q_T}{\Domain}^2}{\kappa}
+
\dfrac{\NormLeb{\Grad s}{\Domain}^2}{2 \mu}
+
\dfrac{\NormLeb{\Grad q_F}{\Domain}^2}{2 \mu}  
\right)^{\frac{1}{2}}. 
\end{equation*} 
By arguing as in the proof of Theorem~\ref{T:cont-infsup}, it follows that the bilinear form involved in problem~\eqref{Biot-four-fields} is uniformly continuous and inf-sup stable, irrespective of all material parameters, provided that the trial norm is defined as
\begin{equation*}
\Big ( 
\Norm{(\widetilde{u}, \widetilde{p}_F)}_1^2
+ \kappa \NormLeb{\lambda \Div(\widetilde u) - \widetilde{p}_T -  \alpha \LebPro{\LebH{\Domain}}(\widetilde p_F)}{\Domain}^2
+ \mu \Norm{\alpha \Div(\widetilde u) -\widetilde{m} +   \sigma \widetilde p_F}_{\SobHD{\Domain}}^2
\Big )^{\frac{1}{2}}.  
\end{equation*} 
Here, compared to \eqref{norm-trial}, we have two additional terms, accounting for the possible relaxation of the constraints \eqref{total-pressure} and \eqref{total-fluid-content}.

\subsection{Guidelines for the discretization}
\label{SS:guidelines}

As mentioned in the introduction, the discretization of the Biot's model is possibly affected by volumetric locking, spurious pressure modes and loss of mass. The previous results and the following informal discussion appear to contribute to the identification of the origin of these undesired effects.

For `large' $\lambda$,  the term $\lambda \Div(u)$, hence the elastic stress tensor $2 \mu \SymGrad(u) + \lambda \Div(u) I$, enters into the stability estimate established in Corollary~\ref{C:stability} only through the $L^2$-norm of the total pressure $p_T$. Moreover, according to the proof of Theorem~\ref{T:cont-infsup}, the presence of the $L^2$-norm of $p_T$ in that estimate hinges on the equivalence \eqref{cont-infsup-stokes}. Therefore, when the discretization of problem~\eqref{Biot-weak} is concerned, we expect that the lack of a discrete counterpart of \eqref{cont-infsup-stokes} possibly results in a poor approximation of the elastic stress tensor, i.e. in volumetric locking.

Similarly, for `small' $\sigma$ and $\kappa$, the stability estimate in Corollary~\ref{C:stability} allows one to control the fluid pressure $p_F$ only through the $L^2$-norm of $p_T$. Hence, by arguing as before, we expect that a discretization of \eqref{Biot-weak} failing to reproduce \eqref{cont-infsup-stokes} is possibly affected by spurious pressure oscillations. This observation and the previous one suggest that
\begin{subequations}
\label{guidelines}
\begin{equation}
\label{guidelines-Stokes}
\begin{minipage}{0.75\hsize}
the displacement $u$ and the total pressure $p_T$ should be discretized by a pair of spaces enjoying a counterpart of \eqref{cont-infsup-stokes}.
\end{minipage}
\end{equation} 
Interestingly, both volumetric locking and spurious pressure oscillations seem to be related to the failure of the same condition, although the nature of the two effects is different, as pointed out in \cite{Haga.Osnes.Langtangen:12}.

Finally, the stability estimate in Corollary~\ref{C:stability} allows one to control also the $H^{-1}$-norm of the total fluid content $m$. By inspecting the proof of Theorem~\ref{T:cont-infsup}, it is clear that this is made possible by the identity \eqref{cont-infsup-riesz}. Thus, we expect that a discretization of \eqref{Biot-weak} is possibly affected by a substantial loss of mass when a counterpart of \eqref{cont-infsup-riesz} fails to hold. This suggests that
\begin{equation}
\label{guidelines-Riesz}
\begin{minipage}{0.75\hsize}
the total fluid content $m$ and the fluid pressure $p_F$ should be discretized by a pair of spaces enjoying a counterpart of \eqref{cont-infsup-riesz}.
\end{minipage}
\end{equation}
\end{subequations}  

The condition \eqref{guidelines-Stokes} is well-known from the discretization of the Stokes equations and several pairs of finite element spaces fulfilling it are known in the literature. In contrast, we are aware of only one (qualitative) result related to the condition \eqref{guidelines-Riesz}, see \cite{Bartels.Wang:20}.

\section{A quasi-optimal and robust discretization}
\label{S:discrete-problem}

In this section we devise and analyze a finite element discretization of problem~\eqref{Biot-weak}, that can be equivalently interpreted as a discretization of \eqref{Biot-four-fields}. For simplicity, we restrict our attention to the lowest-order case. We briefly address the derivation of higher-order discretizations in section~\ref{SS:higher-order}.

\subsection{Simplicial meshes and finite element spaces}
\label{SS:meshes}

Let $\Mesh$ be a face-to-face simplicial mesh of $\Omega$. The shape parameter $\Shape{\Mesh}$ of $\Mesh$ is defined as
\begin{equation*}
\label{shape-parameter}
\Shape{\Mesh} := \max_{\MeshEl \in \Mesh} \dfrac{\mathrm{diam(\MeshEl)}}{\mathrm{diam}(\mathsf{B_\MeshEl})}
\end{equation*}
where $\mathsf{B}_\mathsf{T}$ indicates the largest ball inscribed in a simplex $\MeshEl \in \Mesh$. The broken version $\mathcal{D}_\Mesh$ of a differential operator $\mathcal{D}$ is given by
\begin{equation*}
\label{broken-operator}
\forall \MeshEl \in \Mesh
\qquad
(\mathcal{D} v)_{|\MeshEl} := \mathcal{D}(v_{|\MeshEl})
\end{equation*}
where $v$ is a piecewise smooth function on $\Mesh$. 

The sets $\Faces$ and $\FacesInt$ consist, respectively, of all the faces and of all the interior faces of $\Mesh$. The skeleton $\Skel$ is obtained by taking the union of all the faces of $\Mesh$. The operators
\begin{equation*}
\label{jump-average}
\text{jump} \quad \Jump{\cdot}
\qquad \text{and} \qquad
\text{average} \quad \Avg{\cdot}
\end{equation*}
map piecewise smooth functions on $\Mesh$ into piecewise smooth functions on $\Skel$ and are defined as usual, see, for instance, \cite[Definition~1.17]{DiPietro.Ern:12}. When composing the jump or the average with a broken differential operator, we omit the subscript $\Mesh$, to alleviate the notation, cf. \eqref{form-DG} below. 

We extend the outer normal unit vector $\Normal_{\partial \Domain}$ of $\Domain$ to a piecewise constant vector field $\Normal: \Skel \to \R^d$. For this purpose, we prescribe a normal unit vector $\Normal_\FaceEl$ for each interior face $\FaceEl \in \FacesInt$. The orientation of $\Normal_\FaceEl$ does not affect our subsequent discussion. Then, we set
\begin{equation*}
\label{normal-field}
\Normal_{|\FaceEl} := \Normal_\FaceEl \quad \text{if} \;\; \FaceEl \in \FacesInt
\qquad \text{and} \qquad
\Normal_{|\FaceEl} := \Normal_{\partial \Domain} \quad \text{if} \;\; \FaceEl \in \Faces \setminus \FacesInt.
\end{equation*}
We also consider the following piecewise constant meshsize function $\MeshSize: \Skel \to \R$ on the skeleton of $\Mesh$
\begin{equation*}
\MeshSize_{|\FaceEl} := \mathrm{diam}(\FaceEl) \qquad \forall \FaceEl \in \Faces.
\end{equation*}

For a nonnegative integer $\Degree \geq 0$ and a simplex $\MeshEl \in \Mesh$, the space $\Poly{\Degree}(\MeshEl)$ consists of all polynomials of total degree $\leq \Degree$ on $\MeshEl$. The corresponding space of possibly discontinuous piecewise polynomials over $\Mesh$ is
\begin{equation*}
\PolyPiec{\Degree} := \{ S: \Domain \to \R \mid \forall \MeshEl \in \Mesh \quad S_{|\MeshEl} \in \Poly{\Degree}(\MeshEl) \}.
\end{equation*} 
We shall repeatedly make use also of the one-codimensional subspace
\begin{equation*}
\PolyPiecAvg{\Degree} 
:=
\PolyPiec{\Degree} \cap \LebH{\Domain}
=
\{ S \in \PolyPiec{\Degree} \mid \int_\Domain S =0 \} 
\end{equation*} 
and of the lowest-order Crouzeix-Raviart space with zero boundary values
\begin{equation}
\label{CR-space}
\CR 
:=
\{ S \in \PolyPiec{1} \mid \forall \FaceEl \in \Faces \quad \int_{\FaceEl} \Jump{S} = 0\}.
\end{equation}

\subsection{Finite element discretization}
\label{SS:FE-discretization}

The stability estimate established in Corollary~\ref{C:stability} involves $H^1$-like norms of the displacement $u$ and of the fluid pressure $p_F$, the $L^2$-norm of the total pressure $p_T$ from \eqref{total-pressure} and the $H^{-1}$-norm of the total fluid content $m$ introduced in \eqref{total-fluid-content}. This indicates that, in principle, we may obtain a first-order discretization of the problem~\eqref{Biot-weak} or, equivalently, of \eqref{Biot-four-fields}, by using piecewise affine functions for approximating $u$ and $p_F$ and piecewise constant functions for approximating $p_T$ and $m$. Therefore, owing to the inclusion $p_T \in \LebH{\Domain}$, we look for approximations
\begin{equation}
\label{spaces-pT-m}
P_T \in \PolyPiecAvg{0} \quad \text{of} \quad p_T
\qquad \text{and} \qquad
M \in \PolyPiec{0} \quad \text{of} \quad m.
\end{equation}

The condition \eqref{guidelines-Stokes} discourages us from approximating $u$ by globally continuous piecewise affine functions because, with this choice, a discrete counterpart of \eqref{cont-infsup-stokes} fails to hold on most meshes, cf. \cite[section~8.3.2]{Boffi.Brezzi.Fortin:13}. Instead, the Crouzeix-Raviart pair $\CR^\Dim/\PolyPiecAvg{0}$ is known to fulfill \eqref{guidelines-Stokes}, see Proposition~\ref{P:cont-infsup-auxiliary-discrete} below. Similarly, we avoid the use of continuous piecewise affine functions and of Crouzeix-Raviart functions for approximating $p_F$. In fact, \eqref{guidelines-Riesz} prescribes, in particular, an inf-sup condition that fails to hold on certain meshes, due to the presence of spurious modes, see Figure~\ref{F:spurious-mode}. We refer to \cite[section~3]{Bartels.Wang:20} for a more extensive discussion on the existence of spurious modes when Crouzeix-Raviart functions are concerned. Using discontinuous piecewise affine functions prevents from the existence of spurious modes, according to the inclusion $\PolyPiec{0} \subseteq \PolyPiec{1}$. We provide a quantitative counterpart of this qualitative observation in Proposition~\ref{P:cont-infsup-auxiliary-discrete}. Thus, we look for approximations 
\begin{equation}
\label{spaces-u-pF}
U \in \CR^\Dim \quad \text{of} \quad u
\qquad \text{and} \qquad
P_F \in \PolyPiec{1} \quad \text{of} \quad p_F.
\end{equation}

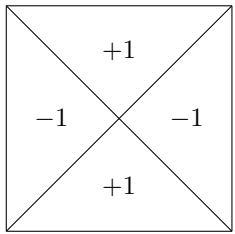
\begin{figure}[ht]
	\centering
	\begin{tikzpicture}
	\coordinate (z1) at (0,0);
	\coordinate (z2) at (3,0);
	\coordinate (z3) at (3,3);
	\coordinate (z4) at (0,3);
	\coordinate (c)  at (1.5,1.5);
	\path (z1) edge (z2);
	\path (z2) edge (z3);
	\path (z3) edge (z4);
	\path (z4) edge (z1);
	\path (z1) edge (z3);
	\path (z2) edge (z4);
	\node at (1.5,0.6) {$+1$};
	\node at (1.5,2.4) {$+1$};
	\node at (0.6,1.5) {$-1$};
	\node at (2.4,1.5) {$-1$};
	\end{tikzpicture}
	\caption{A piecewise constant function that annihilates the $L^2$-scalar product times all the continuous piecewise affine functions and all the Crouzeix-Raviart functions with zero boundary values.}
	\label{F:spurious-mode}
\end{figure}

Having prescribed a finite element space for the approximation of each variable involved in the Biot's model, we aim at introducing a discrete counterpart of the bilinear form $b$ in \eqref{Biot-weak}-\eqref{form-b}. For this purpose, we consider the bilinear form $A_\CRLabel: \CR^\Dim \times \CR^\Dim \to \R$ 
\begin{equation}
\label{form-CR}
A_\CRLabel(\widetilde U, V) 
:= 
\int_\Domain \SymGradM(\widetilde{U}) \colon \SymGradM(V)
+
\int_{\Skel} \MeshSize^{-1} \Jump{\widetilde{U}} \cdot \Jump{V} 
\end{equation}
for all $\widetilde{U}, V \in \CR^\Dim$. This form serves as a replacement of the $L^2$-scalar product of the symmetric gradients. The jump penalization prevents from the lack of a discrete Korn's inequality \cite{Arnold:93} and it is inspired by the results in \cite{Brenner:04}, cf \eqref{Korn-piecewise} below. 

We discretize the $L^2$-scalar product of the gradients by the so-called symmetric interior penalty bilinear form $A_\DGLabel:\PolyPiec{1} \times \PolyPiec{1} \to \R$, that is often employed in discontinuous Galerkin methods, see, e.g., \cite[Chapter~4]{DiPietro.Ern:12}. This form is defined as
\begin{equation}
\label{form-DG}
\begin{split}
A_\DGLabel(\widetilde{P}_F, Q_F) &:=
\int_\Domain \GradM \widetilde{P}_F \cdot \GradM Q_F
-
\int_{\Skel} \Avg{\Grad \widetilde{P}_F} \cdot \Normal \Jump{Q_F}\\
&-
\int_{\Skel} \Jump{\widetilde{P}_F} \Avg{\Grad Q_F} \cdot \Normal
+
\int_{\Skel} \dfrac{\eta}{\MeshSize} \Jump{\widetilde{P}_F} \Jump{Q_F}
\end{split}
\end{equation}
for all $\widetilde{P}_F, Q_F \in \PolyPiec{1}$, where $\eta > 0$ is a penalty parameter to be specified later. 

The inclusions \eqref{spaces-pT-m} suggest to consider the following counterparts
\begin{equation}
\label{pT-m-discrete}
P_T = \lambda \DivM(U) - \alpha \LebPro{\PolyPiecAvg{0}}(P_F)
\qquad \text{and}\qquad
M = \alpha \DivM(U) + \sigma \LebPro{\PolyPiec{0}} (P_F)
\end{equation}
of the relations \eqref{total-pressure} and \eqref{total-fluid-content}, where $\LebPro{\PolyPiec{0}}$ and $\LebPro{\PolyPiecAvg{0}}$ are the $L^2$-orthogonal projections onto $\PolyPiec{0}$ and $\PolyPiecAvg{0}$, respectively.

We obtain a discrete counterpart $B: (\CR^\Dim \times \PolyPiec{0}) \times (\CR^\Dim \times \PolyPiec{0}) \to \R$ of the bilinear form $b$ in \eqref{Biot-weak}-\eqref{form-b} by combining the ingredients listed above, namely
\begin{equation}
\label{form-B}
\begin{alignedat}{1}
B( (\widetilde U, \widetilde{P}_F), (V, Q_F) ) :=
& 2 \mu A_\CRLabel(\widetilde{U}, V) 
+ \int_\Domain (\lambda \DivM(\widetilde U) - \alpha \widetilde{P}_F) \DivM(V)\\
& + \int_\Domain ( \alpha \DivM(\widetilde{U}) + \sigma \LebPro{\PolyPiec{0}}(\widetilde{P}_F) )Q_F + \kappa A_\DGLabel(\widetilde{P}_F, Q_F)
\end{alignedat}
\end{equation}
for all $(\widetilde U, \widetilde P_F), (V, Q_F) \in \CR^\Dim \times \PolyPiec{0}$.

\begin{remark}[Reduced integration]
\label{R:reduced-integration}
We have replaced the piecewise affine function $P_F$ by its $L^2$-orthogonal projection onto piecewise constants in the second part of \eqref{pT-m-discrete}, so as to enforce the inclusion $M \in \PolyPiec{0}$. This has the effect that the $L^2$ scalar product $\int_\Domain \widetilde{p}_F q_F$ in the definition \eqref{form-b} of the form $b$ is replaced by $\int_\Domain \LebPro{\PolyPiec{0}}(\widetilde{P}_F) \LebPro{\PolyPiec{0}}(Q_F)$ in $B$, i.e. it is discretized by a reduced integration technique.  
\end{remark}

The right-hand side of the problem~\eqref{Biot-weak} deserves to be discretized as well. To this end, we cannot just take the restriction of the loads $f$ and $g$ to the spaces $\CR^\Dim$ and $\PolyPiec{1}$, respectively. In fact, these spaces are nonconforming, meaning that, in general, we have
\begin{equation*}
\CR^d \nsubseteq \SobH{\Domain}^\Dim
\qquad \text{and} \qquad 
\PolyPiec{1} \nsubseteq \SobH{\Domain}.
\end{equation*} 
Most often, this issue is dealt with by assuming that $f$ and $g$ are more regular than in \eqref{loads-regularity}, for instance $f \in \Leb{\Domain}^\Dim$ and $g \in \Leb{\Domain}$, so that the products $\int_\Domain f \cdot V$ and $\int_\Domain g Q_F$ are defined for all $V \in \CR^\Dim$ and $Q_F \in \PolyPiec{1}$. We briefly consider this option in section~\ref{SS:medius-analysis} below. Here, inspired by the abstract results in \cite{Veeser.Zanotti:18}, we approach the problem differently. Instead of invoking additional regularity of the data beyond \eqref{loads-regularity}, we introduce two linear operators 
\begin{equation*}
\label{smoothers}
\Smt_\CRLabel : \CR^\Dim \to \SobH{\Domain}^\Dim
\qquad \text{and} \qquad
\Smt_\DGLabel: \PolyPiec{1} \to \SobH{\Domain} 
\end{equation*}
and we observe that the dualities $\left\langle f, \Smt_\CRLabel (V) \right\rangle $ and $\left\langle g, \Smt_\DGLabel (Q_F) \right\rangle $ make sense for all loads $(f, g) \in \SobHD{\Domain; \Dim} \times \SobHD{\Domain}$ and for all test functions $(V, Q_F) \in \CR^\Dim \times \PolyPiec{1}$.

Let us assume for the moment that $\Smt_\CRLabel$ and $\Smt_\DGLabel$ are given. Then, we consider the following discretization of the problem~\eqref{Biot-weak}: 
\begin{equation}
\label{Biot-discrete}
\begin{gathered}[c]
\text{find} \quad (U, P_F)  \in \CR^\Dim \times \PolyPiec{1} \quad \text{such that}\\
\forall (V, Q_F) \in \CR^\Dim \times \PolyPiec{1} 
\quad \; 
B((U, P_F), (V, Q_F))
=
\left\langle f, \Smt_\CRLabel (V)\right\rangle + \left\langle g, \Smt_\DGLabel (Q_F)\right\rangle. 
\end{gathered}
\end{equation}
Proceeding as in Section~\ref{SS:four-fields}, we could combine this problem with the identities in \eqref{pT-m-discrete} and derive a discretization of the four-field formulation \eqref{Biot-four-fields}.

\begin{remark}[Guidelines for $\Smt_\CRLabel$ and $\Smt_\DGLabel$]
\label{R:guidelines-smooothers}
The operators $\Smt_\CRLabel$ and $\Smt_\DGLabel$ must be explicitly computed when assembling the problem~\eqref{Biot-discrete}. Therefore, it is important that their action can be `easily' evaluated. To this end, they should involve, at most, the solution of finite-dimensional local problems. Additionally, one may expect that the size of the continuity constants of $\Smt_\CRLabel$ and $\Smt_\DGLabel$ plays an important role when establishing a counterpart of the stability estimate in Corollary~\ref{C:stability}. The use of the averaging operator in \eqref{smoother} serves to keep such constants under control, as stated by Proposition~\ref{P:smoother-bounded}. Finally, $\Smt_\CRLabel$ and $\Smt_\DGLabel$ should be consistent with the bilinear form $B$, in a sense that could be made rigorous in the vein of \cite[Definition~2.7]{Veeser.Zanotti:18}. We enforce consistency by prescribing the conservation of the lowest-order moments in the simplices and on the interior faces of $\Mesh$, cf. Lemma~\ref{L:moments-preserved}.
\end{remark}

Our construction of the operators $\Smt_\CRLabel$ and $\Smt_\DGLabel$ is inspired by \cite[section~3]{Veeser.Zanotti:19b} and \cite[section~3]{Veeser.Zanotti:18b}. Denote by $\Verts$ and by $\VertsInt$, respectively, the sets collecting all the vertices and all the interior vertices of $\Mesh$. Recall that the space $\SobH{\Domain} \cap \PolyPiec{1}$ consists of continuous piecewise affine functions on $\Mesh$ and that its Lagrange basis $(S_\VertsEl)_{\VertsEl \in \VertsInt}$ is indexed by the interior vertices. A simple strategy to map discontinuous piecewise affine functions into continuous ones consists in averaging the point values around each vertex, cf. \cite[section~5.5.2]{DiPietro.Ern:12}. More precisely, we consider the linear operator $\mathcal{A}: \PolyPiec{1} \to \SobH{\Domain} \cap \PolyPiec{1}$ defined as follows
\begin{equation}
\label{averaging-operator}
\mathcal{A}(S) := 
\sum_{\VertsEl \in \VertsInt} \dfrac{1}{\# \Mesh_\VertsEl} 
\left( \sum_{\MeshEl \in \Mesh_\VertsEl} S_{|\MeshEl}(\VertsEl) \right) S_\VertsEl 
\end{equation} 
for all $S \in \PolyPiec{1}$. Here, the local mesh $\Mesh_\VertsEl$, $\VertsEl \in \VertsInt$, consists of all the simplices of $\Mesh$ touching the vertex $\VertsEl$.

For each interior face $\FaceEl \in \FacesInt$ and for each simplex $\MeshEl \in \Mesh$, we consider the bubble functions
\begin{equation}
\label{bubble-functions}
S_\FaceEl 
:=
\dfrac{(2\Dim-1)!}{(\Dim-1)! |\FaceEl|}
\prod_{\VertsEl \in \Verts \cap \FaceEl} S_\mathsf{z}
\qquad \text{and} \qquad
S_\MeshEl 
:=
\dfrac{(2\Dim+1)!}{\Dim! |\MeshEl|}
\prod_{\VertsEl \in \Verts \cap \MeshEl} S_\mathsf{z}
\end{equation}
where $|\FaceEl|$ and $|\MeshEl|$ are the $(\Dim-1)$-dimensional and the $\Dim$-dimensional Lebesgue measures of $\FaceEl$ and $\MeshEl$, respectively. Functions in this form are widely used, e.g., in the a posteriori analysis of finite element methods, see \cite[section~3.2.3]{Verfuerth:13}. Notice that $S_\FaceEl \in \SobH{\Domain} \cap \PolyPiec{\Dim}$ and  $S_\MeshEl \in \SobH{\Domain} \cap \PolyPiec{\Dim+1}$. Moreover, the functions $S_\FaceEl$ and $S_\MeshEl$ are locally supported and are scaled so that the identities in \eqref{bubble-integral} below hold true.

Let $\Smt: \PolyPiec{1} \to \SobH{\Domain}$ be given by
\begin{equation}
\label{smoother}
\Smt(S) := \mathcal{A}(S) + 
\sum_{\FaceEl \in \FacesInt} \left( \int_\FaceEl (\Avg{S} - \mathcal{A}(S)) \right) S_\FaceEl 
\end{equation}
for all $S \in \PolyPiec{1}$. This operator coincides with the ones in \cite[Lemma~3.3]{Carstensen.Schedensack:15} for $\Dim =2$ and in \cite[Proposition~3.4]{Veeser.Zanotti:18b} for general $\Dim \geq 2$. We define $\Smt_\CRLabel: \CR^\Dim \to \SobH{\Domain}^\Dim$ by
\begin{equation}
\label{smoother-CR}
\Smt_\CRLabel(V) := ( \Smt(V_1), \dots, \Smt(V_\Dim) )
\end{equation}
for all $V = (V_1, \dots, V_\Dim) \in \CR^\Dim$. Furthermore, we define $\Smt_\DGLabel: \PolyPiec{1} \to \SobH{\Domain}$ by
\begin{equation}
\label{smoother-DG}
\Smt_\DGLabel(Q_F) := \Smt(Q_F) + 
\sum_{\MeshEl \in \Mesh}
\left( \int_\MeshEl (Q_F - \Smt(Q_F)) \right) S_\MeshEl 
\end{equation}
for all $Q_F \in \PolyPiec{1}$. 
The `correction' of the averaging operator $\mathcal{A}$ by the bubble functions ensures the validity of the following result.  

\begin{lemma}[Moments preserved by $\Smt_\CRLabel$ and $\Smt_\DGLabel$]
\label{L:moments-preserved}	
Let $V \in \CR^\Dim$ and $Q_F \in \PolyPiec{1}$ be given. The operators $\Smt_\CRLabel$ and $\Smt_\DGLabel$ defined in \eqref{smoother-CR} and \eqref{smoother-DG}, respectively, are such that
\begin{subequations}
\label{moments-preserved}
\begin{equation}
\label{moments-preserved-CR}
\int_\FaceEl \Smt_\CRLabel(V) 
=
\int_\FaceEl V
\end{equation}
as well as
\begin{equation}
\label{moments-preserved-DG}
\int_\FaceEl \Smt_\DGLabel(Q_F) 
= 
\int_\FaceEl \Avg{Q_F}
\qquad \text{and} \qquad
\int_\MeshEl \Smt_\DGLabel(Q_F)
=
\int_\MeshEl Q_F
\end{equation}
\end{subequations}
for all $F \in \FacesInt$ and for all $T \in \Mesh$.
\end{lemma}

\begin{proof}
The bubble function $S_\FaceEl$, $\FaceEl \in \FacesInt$, introduced in \eqref{bubble-functions} is supported on the union of the two simplices sharing $\FaceEl$ and it is continuous in $\Domain$. Furthermore, the scaling factor is chosen so that
\begin{subequations}
\label{bubble-integral}
\begin{equation}
\label{bubble-integral-face}
\int_{\FaceEl^\prime} S_{\FaceEl} = \delta_{\FaceEl \FaceEl^\prime}
\qquad
\forall \FaceEl^\prime \in \Faces.
\end{equation}
Analogously, the bubble function $S_\MeshEl$, $\MeshEl \in \Mesh$, is supported on $\MeshEl$ and it is continuous in $\Domain$. Hence it vanishes on the skeleton $\Skel$ of $\Mesh$. Moreover, it is scaled so that
\begin{equation}
\label{bubble-integral-simplex}
\int_{\MeshEl^\prime} S_{\MeshEl} = \delta_{\MeshEl \MeshEl^\prime}
\qquad
\forall \MeshEl^\prime \in \Mesh.
\end{equation}
\end{subequations}
These observations readily provide \eqref{moments-preserved}, in combination with the definitions of $\Smt_\CRLabel$ and $\Smt_\DGLabel$. Note, in particular, that the integral $\int_\FaceEl V$ is well-defined for all $V \in \CR^\Dim$ and $\FaceEl \in \FacesInt$, although $V$ is not globally continuous in $\Domain$, according to the definition \eqref{CR-space} of the space $\CR$.
\end{proof}

\subsection{Stability of the discretization}
\label{SS:stability-discretization}

Assessing the stability of the problem~\eqref{Biot-discrete} requires some technical preliminaries. Roughly speaking, we need a counterpart of each result invoked in section~\ref{SS:nonsymmetric-setting}. 

First of all, we extend $\NormLeb{\SymGrad(\cdot)}{\Domain}$ to a norm $\Norm{\cdot}_\CRLabel$ on $\SobH{\Domain}^\Dim + \CR^\Dim$ as follows
\begin{equation}
\label{extended-norm-CR}
\Norm{\widetilde{u} + \widetilde{U}}_\CRLabel :=
\left( \NormLeb{\SymGradM(\widetilde{u} + \widetilde{U})}{\Domain}^2 
+ \int_{\Skel} \MeshSize^{-1} |\Jump{\widetilde{U}}|^2 \right)^\frac{1}{2} 
\end{equation}
for all $\widetilde u \in \SobH{\Domain}^\Dim$ and $\widetilde{U} \in \CR^\Dim$. Then, the Korn's inequality
\begin{equation}
\label{Korn-piecewise}
\NormLeb{\GradM(\widetilde{u} + \widetilde{U})}{\Domain}
\lesssim
\Norm{\widetilde{u} + \widetilde{U}}_\CRLabel
\end{equation}
holds true, according to \cite[Theorem~3.1 and Remark~3.3]{Brenner:04}, and the hidden constant only depends on the shape parameter $\Shape{\Mesh}$ of $\Mesh$. 

Similarly, we extend $\NormLeb{\Grad \cdot}{\Domain}$ to a norm $\Norm{\cdot}_\DGLabel$ on $\SobH{\Domain} + \PolyPiec{1}$ as follows
\begin{equation}
\label{extended-norm-DG}
\Norm{\widetilde{p}_F + \widetilde{P}_F}_\DGLabel
:=
\left( \NormLeb{\GradM(\widetilde{p}_F + \widetilde{P}_F)}{\Domain}^2 + \int_{\Skel} \dfrac{\eta}{\MeshSize} |\Jump{\widetilde{P}_F}|^2 \right)^\frac{1}{2} 
\end{equation}
for all $\widetilde{p}_F \in \SobH{\Domain}$ and $\widetilde{P}_F \in \PolyPiec{1}$, where $\eta$ is the penalty parameter involved in the definition \eqref{form-DG} of $A_\DGLabel$. For sufficiently large $\eta$, the form $A_\DGLabel$ is inf-sup stable and bounded with respect to the norm $\Norm{\cdot}_\DGLabel$. More precisely 
\begin{equation}
\label{cont-infsup-DG}
\begin{minipage}{0.9\hsize}
there is $\overline{\eta} > 0$ such that $\forall \eta > \overline{\eta}, \;\widetilde{P}_F \in \PolyPiec{1} \quad
\displaystyle \sup_{Q_F \in \PolyPiec{1}} \dfrac{A_\DGLabel(\widetilde{P}_F, Q_F)}{\Norm{Q_F}_\DGLabel} \approx \Norm{\widetilde{P}_F}_\DGLabel$
\end{minipage}
\end{equation}
see \cite[Lemmas~4.12, 4.16 and 4.20]{DiPietro.Ern:12}.
Both $\overline{\eta}$ and the hidden constants only depend on $\Shape{\Mesh}$. 

The operator $\Smt$ introduced in \eqref{smoother} is bounded in the norm $\Norm{\cdot}_\DGLabel$. In fact, for all $S \in \PolyPiec{1}$ and $\MeshEl \in \Mesh$, we have the local estimate
\begin{equation}
\label{smoother-local-estimate}
\NormLeb{\Grad(S- \Smt(S))}{\MeshEl}
\lesssim
\sum_{\FaceEl \cap \MeshEl \neq \emptyset}
\left( \int_{\FaceEl} 
\MeshSize^{-1} |\Jump{S}|^2\right)^{\frac{1}{2}} 
\end{equation}
where $\FaceEl$ varies in $\Faces$ and the hidden constant only depends on $\Shape{\Mesh}$. This result is proved in \cite[Proposition~3.4 and Eqs. (3.17) and (3.18)]{Veeser.Zanotti:18b} in a slightly different setting (a variant of the averaging \eqref{averaging-operator} is considered) but it holds true also in this case. The boundedness of $\Smt$ implies that the operators $\Smt_\CRLabel$ and $\Smt_\DGLabel$ involved in problem~\eqref{Biot-discrete} are bounded as well.

\begin{proposition}[Boundedness of $\Smt_\CRLabel$ and $\Smt_\DGLabel$]
\label{P:smoother-bounded}
The operators $\Smt_\CRLabel$ and $\Smt_\DGLabel$ defined in \eqref{smoother-CR} and \eqref{smoother-DG}, respectively, are such that
\begin{subequations}
\label{smoother-bounded}
\begin{equation}
\label{smoother-CR-bounded}
\Norm{\Smt_\CRLabel(V)}_\CRLabel \lesssim \Norm{V}_\CRLabel
\end{equation}
as well as
\begin{equation}
\label{smoother-DG-bounded}
\Norm{\Smt_\DGLabel(Q_F)}_\DGLabel \lesssim \Norm{Q_F}_\DGLabel
\qquad \text{and} \qquad
\NormLeb{\Smt_\DGLabel(Q_F)}{\Domain} \lesssim \NormLeb{Q_F}{\Domain}
\end{equation}
\end{subequations}
for all $V \in \CR^\Dim$ and $Q_F \in \PolyPiec{1}$. Moreover, all the hidden constants only depend on the shape parameter $\Shape{\Mesh}$ of $\Mesh$.  
\end{proposition}  

\begin{proof}
Let $V \in \CR^\Dim$ and $\MeshEl \in \Mesh$ be given. The definition \eqref{smoother-CR} and the estimate \eqref{smoother-local-estimate} imply that
\begin{equation*}
\NormLeb{\Grad \Smt_\CRLabel(V)}{\MeshEl} 
\lesssim
\NormLeb{\Grad V}{\MeshEl} 
+ \sum_{\FaceEl \cap \MeshEl \neq \emptyset}
\left(\int_{\FaceEl} \MeshSize^{-1} |\Jump{V}|^2 \right)^{\frac{1}{2}}  
\end{equation*}	
where $\FaceEl$ varies in $\Faces$. Summing over all simplices $\MeshEl \in \Mesh$, we obtain
\begin{equation*}
\Norm{\Smt_\CRLabel(V)}_\CRLabel
\lesssim
\NormLeb{\Grad \Smt_\CRLabel(V)}{\Domain}
\lesssim\left( \NormLeb{\GradM V}{\Domain}^2 + 
\int_{\Skel} \MeshSize^{-1} |\Jump{V}|^2 \right)^{\frac{1}{2}}. 
\end{equation*}
Then, we derive the first claimed inequality \eqref{smoother-CR-bounded} by invoking the Korn's inequality \eqref{Korn-piecewise}. Next, let $Q_F \in \PolyPiec{1}$ and $\MeshEl \in \Mesh$ be given. The definition \eqref{smoother-DG} entails that
\begin{equation*}
(Q_F - \Smt_\DGLabel(Q_F))_{|\MeshEl} 
=
(Q_F - \Smt(Q_F))_{|\MeshEl}
+
\left( \int_\MeshEl (Q_F - \Smt(Q_F)) \right) S_\MeshEl 
\end{equation*}
because each bubble $S_{\MeshEl'}$ vanishes in $\MeshEl$ for $\MeshEl' \neq \MeshEl$. A standard scaling argument reveals that $\NormLeb{\Grad S_\MeshEl}{\MeshEl} \lesssim \mathrm{diam}(\MeshEl)^{-1} |\MeshEl|^{-\frac{1}{2}}$ and $\NormLeb{S_\MeshEl}{\MeshEl} \lesssim |\MeshEl|^{-\frac{1}{2}}$. Hence, we see that 
\begin{equation*}
\begin{split}
&\mathrm{diam}(\MeshEl) \NormLeb{\Grad(Q_F - \Smt_\DGLabel(Q_F))}{\MeshEl}
+
\NormLeb{Q_F - \Smt_\DGLabel(Q_F)}{\MeshEl}
\\
& \qquad \qquad \lesssim \mathrm{diam}(\MeshEl) \NormLeb{\Grad(Q_F - \Smt(Q_F))}{\MeshEl} + \NormLeb{Q_F - \Smt(Q_F)}{\MeshEl}\\
&\qquad \qquad \lesssim \mathrm{diam}(\MeshEl) \NormLeb{\Grad(Q_F - \Smt(Q_F))}{\MeshEl} + 
\dfrac{|\MeshEl|^\frac{1}{2}}{|\partial \MeshEl|}\sum_{\FaceEl \subseteq \partial \MeshEl}\left| \int_\FaceEl (Q_F - \Smt(Q_F))_{|\MeshEl} \right| 
\end{split}
\end{equation*}
where $\FaceEl$ varies in $\Faces$ and the second inequality follows from a scaled Poincar\'{e} inequality \cite[Lemma~B.63]{Ern.Guermond:04}. For an interior face $\FaceEl \in \FacesInt$, the first part of \eqref{moments-preserved-DG} entails that we have $|\int_{\FaceEl} (Q_F - \Smt(Q_F))_{|\MeshEl}| = |\int_{\FaceEl} (Q_F - \Avg{Q_F})_{|\MeshEl}| =  |\int_\FaceEl\Jump{Q_F}|/2$. Similarly, if $\FaceEl$ is a boundary face, i.e. $\FaceEl \in \Faces \setminus \FacesInt$, it holds that $\int_{\FaceEl} (Q_F - \Smt(Q_F))_{|\MeshEl} = \int_\FaceEl (Q_F)_{|\MeshEl} = \int_\FaceEl \Jump{Q_F}$. We insert these identities and \eqref{smoother-local-estimate} into the previous inequality. It follows that
\begin{equation*}
\mathrm{diam}(\MeshEl) \NormLeb{\Grad(Q_F - \Smt_\DGLabel(Q_F))}{\MeshEl}
+
\NormLeb{Q_F - \Smt_\DGLabel(Q_F)}{\MeshEl}
\lesssim 
\sum_{\FaceEl \cap \MeshEl \neq \emptyset}
\left(\int_{\FaceEl}  \MeshSize |\Jump{Q_F}|^2
\right)^{\frac{1}{2}}. 
\end{equation*} 
We derive that the first part of \eqref{smoother-DG-bounded} holds true dividing by $\mathrm{diam}(\MeshEl)$ and summing over all $\MeshEl \in \Mesh$. Regarding the second part of \eqref{smoother-DG-bounded}, we sum over all $\MeshEl \in \Mesh$ and then we observe that
\begin{equation*}
\sum_{\MeshEl \in \Mesh} \sum_{\FaceEl \cap \MeshEl \neq \emptyset} \int_{\FaceEl} \MeshSize |\Jump{Q_F}|^2
\lesssim
\int_{\Skel} \MeshSize |\Jump{Q_F}|^2
\lesssim
\NormLeb{Q_F}{\Domain}
\end{equation*}  
according to the inverse trace inequality \cite[Lemma~1.46]{DiPietro.Ern:12}.
\end{proof}

We are now in position to prove that the finite element spaces chosen for the discretization of the problem~\eqref{Biot-discrete} fulfill the conditions \eqref{guidelines}. In other words, we establish suitable counterparts of the equivalences stated in \eqref{cont-infsup-auxiliary}.

\begin{proposition}[Continuity and inf-sup stability of two auxiliary forms]
\label{P:cont-infsup-auxiliary-discrete}
For all functions $Q_0 \in \PolyPiecAvg{0}$ and $Q \in \PolyPiec{0}$, it holds that
\begin{equation}
\label{cont-infsup-auxiliary-discrete}
\sup_{V \in \CR^\Dim} 
\dfrac{\int_\Domain Q_0 \DivM(V)}{\Norm{V}_\CRLabel}
\approx 
\NormLeb{Q_0}{\Domain}
\qquad \text{and} \qquad
\sup_{Q_F \in \PolyPiec{1}}
\dfrac{\int_\Domain Q Q_F}{\Norm{Q_F}_\DGLabel} 
\approx 
\Norm{Q}_{\SobHD{\Domain}}
\end{equation}
and the hidden constants only depend on the shape parameter $\Shape{\Mesh}$ of $\Mesh$.
\end{proposition}

\begin{proof}
The discussion in \cite[section~8.4.4]{Boffi.Brezzi.Fortin:13} shows that $\CR^\Dim / \PolyPiecAvg{0}$ is a stable pair for the approximation of the Stokes equations. In fact, it holds that
\begin{equation*}
\sup_{V \in \CR^\Dim} 
\dfrac{\int_\Domain Q_0 \DivM(V)}{\NormLeb{\GradM V}{\Domain}}
\approx 
\NormLeb{Q_0}{\Domain}
\end{equation*}
for all $Q_0 \in \PolyPiecAvg{0}$. (Actually, the hidden constants only depend on $\Domain$, but this observation is not relevant here.) The first part of \eqref{cont-infsup-auxiliary-discrete} follows by combining this equivalence with the Korn's inequality \eqref{Korn-piecewise}. Next, recall the $L^2$-orthogonal projection $\LebPro{\PolyPiec{0}}$ onto $\PolyPiec{0}$. By definition, we see that $\LebPro{\PolyPiec{0}}$ is a Fortin operator for the bilinear form involved in the second part of \eqref{cont-infsup-auxiliary-discrete}. Moreover, the trace inequality \cite[Lemma~1.49]{DiPietro.Ern:12} and a scaled version of the Poincar\'{e} inequality \cite[Lemma~B.63]{Ern.Guermond:04} reveal that $\LebPro{\PolyPiec{0}}$ is bounded in the norm $\Norm{\cdot}_\DGLabel$
\begin{equation}
\label{L2-projectio-H1-bounded}
\Norm{\LebPro{\PolyPiec{0}}(q_F)}_\DGLabel
\lesssim
\NormLeb{\Grad q_F}{\Domain} + 
\left( \sum_{\MeshEl \in \Mesh} \mathrm{diam}(\MeshEl)^{-2} \NormLeb{q_F - \LebPro{\PolyPiec{0}}(q_F)}{\MeshEl} \right)^{\frac{1}{2}}
\lesssim
\NormLeb{\Grad q_F}{\Domain}
\end{equation}
for all $q_F \in \SobH{\Domain}$. Let $Q \in \PolyPiec{0}$. The identity \eqref{cont-infsup-riesz} and this estimate entail that
\begin{equation*}
\Norm{Q}_{\SobHD{\Domain}}
=
\sup_{q_F \in \SobH{\Domain}}
\dfrac{\int_\Domain Q \LebPro{\PolyPiec{0}}(q_F)}{\NormLeb{\Grad q_F}{\Domain}}
\lesssim
\sup_{Q_F \in \PolyPiec{0}}
\dfrac{\int_\Domain Q Q_F}{\Norm{Q_F}_\DGLabel}
\leq
\sup_{Q_F \in \PolyPiec{1}}
\dfrac{\int_\Domain Q Q_F}{\Norm{Q_F}_\DGLabel}.
\end{equation*}
The converse of this inequality follows from the fact that the operator $\Smt_\DGLabel$ introduced in \eqref{smoother-DG} is a bounded right inverse of $\LebPro{\PolyPiec{0}}$. Indeed, by recalling the second part of \eqref{moments-preserved-DG}, we see that
\begin{equation*}
\int_\Domain Q Q_F 
=
\int_\Domain Q \Smt_\DGLabel(Q_F)
\leq 
\Norm{Q}_{\SobHD{\Domain}}
\Norm{\Smt_\DGLabel(Q_F)}_\DGLabel
\end{equation*}
for all $Q_F \in \PolyPiec{1}$. Therefore, the first part of \eqref{smoother-DG-bounded} yields
\begin{equation*}
\dfrac{\int_\Domain QQ_F}{\Norm{Q_F}_\DGLabel}
\lesssim
\Norm{Q}_{\SobHD{\Domain}}.
\end{equation*}
We conclude by taking the supremum over all $Q_F \in \PolyPiec{1}$.
\end{proof}

\begin{remark}[Alternative equivalence]
\label{R:sharpened-continuity-infsup}
The equivalences stated in \eqref{cont-infsup-auxiliary-discrete} are motivated by our choice of the finite element spaces for the discretization of the problem \eqref{Biot-weak} and are tailored to our subsequent analysis. Still, the proof of Proposition~\ref{P:cont-infsup-auxiliary-discrete} reveals that we may replace the second part of \eqref{cont-infsup-auxiliary-discrete} by
\begin{equation*}
\forall Q \in \PolyPiec{0}
\qquad
\sup_{Q_F \in \PolyPiec{0}}
\dfrac{\int_\Domain QQ_F}{\Norm{Q_F}_\DGLabel}
\approx
\Norm{Q}_{\SobHD{\Domain}}
\end{equation*}	
where the hidden constants only depend on the shape parameter $\Shape{\Mesh}$ of $\Mesh$.
\end{remark}

The equivalences established in Proposition~\ref{P:cont-infsup-auxiliary-discrete} allow us to state the main result of this section, that is a counterpart of Theorem~\ref{T:cont-infsup} for the bilinear form $B$. For this purpose, we assume hereafter that
\begin{equation}
\label{penalty-parameter}
\begin{minipage}{0.75\hsize}
the penalty parameter $\eta$ in the definition \eqref{form-DG} of $A_\DGLabel$ fulfills the condition $\eta > \overline{\eta}$, where $\overline{\eta}$ is as in \eqref{cont-infsup-DG}.
\end{minipage}
\end{equation}
Furthermore, we replace the test norm $\Norm{\cdot}_2$ introduced in \eqref{norm-test} by
\begin{equation}
\label{norm-test-discrete}
\Normtr{(V, Q_F)}_2 := 
\left( \dfrac{\Norm{V}_{\CRLabel}^2}{\kappa} + 
\dfrac{\Norm{Q_F}_{\DGLabel}^2}{2 \mu} \right)^{\frac{1}{2}}   
\end{equation}
for all test functions $(V, Q_F) \in \CR^\Dim \times \PolyPiec{1}$.
 
\begin{theorem}[Continuity and inf-sup stability of $B$]
	\label{T:cont-infsup-discrete}
	Assume that \eqref{penalty-parameter} holds true and let the form $B$ be defined by \eqref{form-B}. For all $(\widetilde{U}, \widetilde{P}_F) \in \CR^\Dim \times \PolyPiec{1}$, we have
	\begin{equation}
	\label{cont-infsup-discrete}
	\sup_{(V, Q_F) \in \CR^\Dim \times \PolyPiec{1}}
	\dfrac{B( (\widetilde{U}, \widetilde{P}_F), (V, Q_F) )}{\Normtr{(V, Q_F)}_2} 
	\approx
	\Normtr{(\widetilde{U}, \widetilde{P}_F)}_1
	\end{equation}
	where the hidden constants only depend on the shape paramter $\Shape{\Mesh}$ of $\Mesh$, the norm $\Normtr{\cdot}_2$ is as in \eqref{norm-test-discrete} and the norm $\Normtr{\cdot}_1$ is given by
	\begin{equation}
	\label{norm-trial-discrete}
	\begin{alignedat}{2}
	\Normtr{(\widetilde{U}, \widetilde{P}_F)}_1 := 
	\Big (  &\mu^2\kappa \Norm{\widetilde U}_\CRLabel^2 
	+ \kappa \NormLeb{\lambda \DivM(\widetilde U) - \alpha \LebPro{\PolyPiecAvg{0}}(\widetilde P_F)}{\Domain}^2\\
	+ &\mu \kappa^2 \Norm{\widetilde P_F}_\DGLabel^2 
	+  \mu \Norm{\alpha \DivM(\widetilde U) + \sigma \LebPro{\PolyPiec{0}}(\widetilde P_F)}_{\SobHD{\Domain}}^2
	\Big )^{\frac{1}{2}}.  
	\end{alignedat}
	\end{equation} 
	Finally, $\LebPro{\PolyPiecAvg{0}}$ and $\LebPro{\PolyPiec{0}}$ denote the $L^2$-orthogonal projections onto $\PolyPiecAvg{0}$ and $\PolyPiec{0}$.
\end{theorem}

\begin{proof}
The proof of the upper bound `$\lesssim$' in \eqref{cont-infsup-discrete} is immediate, whereas the proof of the lower bound `$\gtrsim$' is similar to the corresponding one in Theorem~\ref{T:cont-infsup}, so we only outline the argument. We introduce the operators $\Riesz_\CRLabel: \PolyPiecAvg{0} \to \CR^\Dim$ and $\Riesz_\DGLabel: \PolyPiec{0} \to \PolyPiec{1}$ through the problems
\begin{equation*}
\begin{alignedat}{2}
& \forall V \in \CR^\Dim\qquad
A_\CRLabel(\Riesz_\CRLabel(Q_0), V) 
&=& \int_\Domain Q_0 \DivM(V)\\
& \forall Q_F \in \PolyPiec{1} \qquad
A_\DGLabel(\Riesz_\DGLabel(Q), Q_F)
&=& \int_\Domain Q Q_F 
\end{alignedat}
\end{equation*}
for all $Q_0 \in \PolyPiecAvg{0}$ and $Q \in \PolyPiec{0}$. The equivalences in Proposition~\ref{P:cont-infsup-auxiliary-discrete} ensure that
\begin{equation}
\label{Rdiv-Riesz-norm-discrete}
\Norm{\Riesz_\CRLabel(Q_0)}_\CRLabel 
\approx
\NormLeb{Q_0}{\Domain}
\qquad \text{and} \qquad
\Norm{\Riesz_\DGLabel(Q)}_\DGLabel
\approx
\Norm{Q}_{\SobHD{\Domain}}.
\end{equation}
For $(\widetilde U, \widetilde{P}_F), (V, Q_F) \in \CR^\Dim \times \PolyPiec{1}$, we rewrite the form $B$ as  
\begin{equation*}
\begin{alignedat}{1}
B( (\widetilde U, \widetilde{P}_F), (V, Q_F) ) =
& A_\CRLabel( 2 \mu \widetilde U +
\Riesz_\CRLabel(\lambda \DivM(\widetilde U) - \alpha \LebPro{\PolyPiecAvg{0}}(\widetilde{P}_F)),  V)\\
+&  A_\DGLabel( \Riesz_\DGLabel(\alpha \DivM(\widetilde{U}) + \sigma \LebPro{\PolyPiec{0}}(\widetilde{P}_F) ) + \kappa \widetilde{P}_F, Q_F).
\end{alignedat}
\end{equation*}
The definition of the norm $\Normtr{\cdot}_2$ entails that
\begin{equation*}
\begin{split}
&\sup_{(V, Q_F) \in \CR^\Dim \times \PolyPiec{1}}
\dfrac{B( (\widetilde{U}, \widetilde{P}_F), (V, Q_F) )}{\Normtr{(V, Q_F)}_2} 
=
\Big( \kappa \Norm{  2\mu \widetilde{U} + \Riesz_\CRLabel(\lambda \DivM(\widetilde U) - \alpha \LebPro{\PolyPiecAvg{0}}(\widetilde{P}_F)) }_\CRLabel^2 \\
&\hspace{140pt}+ 
2 \mu \Norm{ \Riesz_\DGLabel(\alpha \DivM(\widetilde{U}) + \sigma \LebPro{\PolyPiec{0}}(\widetilde{P}_F)) + \kappa \widetilde P_F }_\DGLabel^2 
\Big)^{\frac{1}{2}} .
\end{split}
\end{equation*}
Straight-forward computations, the definition of the operators $\Riesz_\CRLabel$ and $\Riesz_\DGLabel$ and the equivalences in \eqref{Rdiv-Riesz-norm-discrete} confirm that the right-hand side is bounded from below by $\Normtr{(\widetilde{U}, \widetilde{P}_F)}_1$. 
\end{proof}

Theorem~\ref{T:cont-infsup-discrete} yields the following result concerning the stability of the problem~\eqref{Biot-discrete}.

\begin{corollary}[Discrete stability]
\label{C:stability-discrete} 
Assume that \eqref{penalty-parameter} holds true and let the form $B$ be defined by \eqref{form-B}. For all load terms $(f, g) \in \SobHD{\Domain; \Dim} \times\SobHD{\Domain}$, the problem~\eqref{Biot-discrete} is uniquely solvable and its solution $(U, P_F) \in \CR^\Dim \times \PolyPiec{1}$ is such that
\begin{equation}
\label{stability-discrete}
\Normtr{(U, P_F)}_1
\approx
\sup_{(V, Q_F) \in \CR^\Dim \times \PolyPiec{1}}
\dfrac{\left\langle f, \Smt_\CRLabel (V) \right\rangle + \left\langle g, \Smt_\DGLabel (Q_F)\right\rangle }{\Normtr{(V, Q_F)}_2}
\lesssim
\Norm{(f, g)}_{2, \star}
\end{equation}
where the hidden constants only depend on the shape parameter $\Shape{\Mesh}$ of $\Mesh$ and the norms $\Normtr{\cdot}_1$, $\Normtr{\cdot}_2$ and $\Norm{\cdot}_{2, \star}$ are as in Theorem~\ref{T:cont-infsup-discrete}, \eqref{norm-test-discrete} and \eqref{norm-test-dual}, respectively.
\end{corollary}

\begin{proof}
We proceed as in the proof of Corollary~\ref{C:stability} and we additionally derive the second part of \eqref{stability-discrete} by Proposition~\ref{P:smoother-bounded}.
\end{proof}

\begin{remark}[Spurious pressure oscillations]
\label{R:spurious-pressure-modes}
When the parameter $\kappa$ is `small', the norm $\Normtr{\cdot}_1$ controls only the $L^2$-orthogonal projection of the discrete fluid pressure $P_F$ onto the piecewise constant functions (and not $P_F$ itself). Therefore, it must be expected that the components of $P_F$ in the $L^2$-orthogonal complement of the piecewise constants are possibly affected by spurious oscillations. This follows from the fact that the space chosen for approximating the total pressure, i.e. $\PolyPiec{0}$, is smaller than the space used for approximating the fluid pressure, i.e. $\PolyPiec{1}$ and from the observation that the pair $\CR^\Dim /\PolyPiec{1}$ does not enjoy a counterpart of the equivalence \eqref{cont-infsup-stokes}. This shows, incidentally, that the condition \eqref{guidelines-Stokes} guarantees that only some projection of the approximate fluid pressure is free from spurious oscillations.	
\end{remark}

\subsection{Error analysis}
\label{SS:error-analysis}

Assessing the quality of the discretization in \eqref{Biot-discrete} requires an error notion on the sum of the spaces $\SobH{\Domain}^\Dim \times \SobH{\Domain}$ and $\CR^\Dim \times \PolyPiec{1}$. The expression of the norms $\Norm{\cdot}_1$ and $\Normtr{\cdot}_1$ in Theorems~\ref{T:cont-infsup} and \ref{T:cont-infsup-discrete}, respectively, suggests to proceed as follows. First, in order to alleviate the notation, we define
\begin{equation}
\label{auxiliary-variables}
\begin{alignedat}{1}
\widetilde{p}_T := \lambda \Div(\widetilde u) - 
\alpha \LebPro{\LebH{\Domain}}(\widetilde{p}_F)
\qquad & \text{and}  \qquad
\widetilde{m} := \alpha \Div(\widetilde u) +
\sigma \widetilde p_F\\
\widetilde{P}_T := \lambda \DivM(\widetilde{U}) -
\alpha \LebPro{\PolyPiecAvg{0}}(\widetilde P_F)
\qquad & \text{and}  \qquad
\widetilde{M} = \alpha \DivM(\widetilde{U}) +
\sigma \LebPro{\PolyPiec{0}} (\widetilde{P}_F) 
\end{alignedat}
\end{equation}
for all $(\widetilde{u}, \widetilde{p}_F) \in \SobH{\Domain}^\Dim \times \SobH{\Domain}$ and $(\widetilde{U}, \widetilde{P}_F) \in \CR^\Dim \times \PolyPiec{1}$. Then, we set
\begin{equation}
\label{error-notion}
\begin{split}
\Err{(\widetilde{u}, \widetilde{p}_F), (\widetilde{U}, \widetilde P_F)}
:= 
\Big (  \mu^2\kappa &\Norm{\widetilde{u} - \widetilde U}_\CRLabel^2 
+ \kappa \NormLeb{\widetilde p_T - \widetilde{P}_T}{\Domain}^2\\
+  \mu &\Norm{\widetilde{m} - \widetilde{M}}_{\SobHD{\Domain}}^2
+ \mu \kappa^2 \Norm{\widetilde{p}_F - \widetilde P_F}_\DGLabel^2 
\Big)^{\frac{1}{2}}.
\end{split}
\end{equation}
Notice that we cannot just define the error through a norm that extends both $\Norm{\cdot}_1$ and $\Normtr{\cdot}_1$, because the two norms act differently on the intersection of the respective spaces. This is ultimately due to the use of a reduced integration technique in the definition of the problem \eqref{Biot-discrete}, cf. Remark~\ref{R:reduced-integration}. 

In addition to the stability observed in the previous section, the error analysis requires also some consistency, i.e. some compatibility between the form $B$ on the left-hand side of the problem~\eqref{Biot-discrete} and the operators $\Smt_\CRLabel$ and $\Smt_\DGLabel$ on the right-hand side. The property of $\Smt_\CRLabel$ and $\Smt_\DGLabel$ ensuring that we indeed have the necessary consistency is the conservation of the moments stated in Lemma~\ref{L:moments-preserved}. 

\begin{proposition}[Consistency]
\label{P:consistency-error}
Let the load term $(f, g) \in \SobHD{\Domain; \Dim} \times \SobHD{\Domain}$ be given and denote by $(u, p_F) \in \SobH{\Domain}^\Dim \times \SobH{\Domain}$ and $(U, P_F) \in \CR^\Dim \times \PolyPiec{1}$ the corresponding solutions of the problems \eqref{Biot-weak} and \eqref{Biot-discrete}, respectively. Then, we have
\begin{equation*}
\label{consistency-error}
\begin{split}
\sup_{(V, Q_F) \in \CR^\Dim \times \PolyPiec{1}}
\dfrac{B((U- \widetilde{U}, P_F - \widetilde{P}_F), (V, Q_F))}{\Normtr{(V,Q_F)}_2}
\lesssim
\Err{(u,p_F), (\widetilde{U}, \widetilde P_F)} 
\end{split}
\end{equation*}
for all $(\widetilde{U}, \widetilde{P}_F) \in \CR^\Dim \times \PolyPiec{1}$, where $B$, $\ERR$ and $\Normtr{\cdot}_2$ are as in \eqref{form-B}, \eqref{error-notion} and \eqref{norm-test-discrete} and the hidden constant only depends on the shape parameter $\Shape{\Mesh}$ of $\Mesh$.
\end{proposition}

\begin{proof}
By comparing the problems~\eqref{Biot-weak} and \eqref{Biot-discrete}, we see that
\begin{equation}
\label{consistency-error-proof}
\begin{split}
B( (U, P_F), (V, Q_F) ) = 
2 \mu &\int_\Domain \SymGrad(u) \colon \SymGrad(\Smt_\CRLabel(V))
+
\int_\Domain p_T \Div(\Smt_\CRLabel(V))\\
+
&\int_\Domain m \Smt_\DGLabel(Q_F)
+\kappa \int_\Domain \Grad p_F \cdot \Grad \Smt_\DGLabel(Q_F).
\end{split}
\end{equation}
We aim at establishing a similar identity for $B((\widetilde U, \widetilde P_F), (V, Q_F))$. Integrating by parts piecewise, we see that
\begin{equation*}
\int_\Domain \SymGradM(\widetilde{U}) \colon \SymGradM(V)
=
\sum_{\MeshEl \in \Mesh} \int_{\partial \MeshEl}
\SymGradM(\widetilde{U}) \Normal_\MeshEl \cdot V
=
\int_{\Skel \cap \Domain} \Jump{\SymGradM(\widetilde{U})} \Normal \cdot V
\end{equation*}
where $\Normal_\MeshEl$ is the outer normal unit vector of $\MeshEl$ and the second identity follows from the definition \eqref{CR-space} of the space $\CR$. Since $\Jump{ \SymGradM(\widetilde{U})} \Normal$ is piecewise constant on $\Skel$, we apply \eqref{moments-preserved-CR} and we integrate back by parts 
\begin{equation*}
\int_\Domain \SymGradM(\widetilde{U}) \colon \SymGradM(V)
=
\int_{\Skel \cap \Domain} \Jump{\SymGradM(\widetilde{U})} \Normal \cdot \Smt_\CRLabel(V)
=
\int_\Domain \SymGradM(\widetilde{U}) \colon \SymGrad(\Smt_\CRLabel(V)).
\end{equation*}
The same argument entails that
\begin{equation*}
\int_\Domain \widetilde{P}_T \DivM(V)
=
\int_{\Skel \cap \Domain} \Jump{\widetilde{P}_T} V \cdot \Normal
=
\int_{\Skel \cap \Domain} \Jump{\widetilde{P}_T} \Smt_\CRLabel(V) \cdot \Normal
=
\int_\Domain \widetilde{P}_T \Div(\Smt_\CRLabel(V)).
\end{equation*}
We exploit the identities in \eqref{moments-preserved-DG} in a similar fashion. Hence, we obtain
\begin{equation*}
\begin{split}
B((\widetilde U, \widetilde P_F), (V, Q_F)) = 
2 \mu \left( \int_\Domain \SymGradM(\widetilde U) \colon \SymGrad(\Smt_\CRLabel(V))
+
\int_{\Skel} \MeshSize^{-1} \Jump{\widetilde{U}} \cdot \Jump{V} \right) \\
+
\int_\Domain \widetilde{P}_T \Div(\Smt_\CRLabel(V))
+
\int_\Domain \widetilde M \Smt_\DGLabel(Q_F)\\
+\kappa \left( \int_\Domain \GradM \widetilde P_F \cdot \Grad \Smt_\DGLabel(Q_F) 
-
\int_{\Skel} \Jump{\widetilde{P}_F} \Avg{\Grad Q_F} \cdot \Normal
+
\int_{\Skel} \dfrac{\eta}{\MeshSize} \Jump{\widetilde{P}_F} \Jump{Q_F}\right). 
\end{split}
\end{equation*}
We compare this identity with \eqref{consistency-error-proof}, then we apply the Cauchy-Schwartz inequality and the inverse estimate $\int_\Skel \Jump{\widetilde{P}_F} \Avg{\Grad Q_F}\cdot \Normal \lesssim ( \int_\Skel \MeshSize^{-1} |\Jump{\widetilde{P}_F}|^2 )^{\frac{1}{2}} \NormLeb{\GradM Q_F}{\Domain}$, cf. \cite[Lemma~1.46]{DiPietro.Ern:12}. It follows that
\begin{equation*}
\begin{split}
B((U- \widetilde{U}, P_F - \widetilde{P}_F), (V, Q_F))
&\lesssim
( 2 \mu \Norm{u- \widetilde{U}}_\CRLabel + \NormLeb{p_T - \widetilde{P}_T}{\Domain}) \Norm{\Smt_\CRLabel(V)}_\CRLabel\\
+ & 2 \mu \Norm{u- \widetilde{U}}_\CRLabel \Norm{V}_\CRLabel + \kappa \Norm{p_F - \widetilde{P}_F}_\DGLabel \Norm{Q_F}_\DGLabel\\
+ & 
(\kappa \Norm{p_F - \widetilde{P}_F}_\DGLabel + \Norm{m - \widetilde{M}}_{\SobHD{\Domain}} ) \Norm{\Smt_\DGLabel(Q_F)}_\DGLabel.
\end{split}
\end{equation*} 
We conclude by invoking the boundedness of $\Smt_\CRLabel$ and $\Smt_\DGLabel$ stated in Proposition~\ref{P:smoother-bounded} and by recalling the definitions of $\Err{\cdot}$ and $\Normtr{\cdot}_2$.
\end{proof}

Theorem~\ref{T:cont-infsup-discrete} and Proposition~\ref{P:consistency-error}, i.e. stability and consistency, readily entail that the problem~\eqref{Biot-discrete} is a quasi-optimal discretization of \eqref{Biot-weak}, with respect to the error notion $\ERR$.

\begin{theorem}[Quasi-optimality]
\label{T:quasi-optimality} 
Assume that \eqref{penalty-parameter} holds true. Let the load term $(f, g) \in \SobHD{\Domain; \Dim} \times \SobHD{\Domain}$ be given and denote by $(u, p_F) \in \SobH{\Domain}^\Dim \times \SobH{\Domain}$ and $(U, P_F) \in \CR^\Dim \times \PolyPiec{1}$ the corresponding solutions of the problems \eqref{Biot-weak} and \eqref{Biot-discrete}, respectively. Then, we have
\begin{equation}
\label{quasi-optimality}
\Err{(u,p_F), (U,P_F)} 
\lesssim 
\inf_{(\widetilde{U}, \widetilde{P}_F) \in \CR^\Dim \times \PolyPiec{1}}\Err{(u,p_F), (\widetilde U, \widetilde P_F)}
\end{equation}
where $\ERR$ is as in \eqref{error-notion} and the hidden constant only depends on the shape parameter $\Shape{\Mesh}$ of $\Mesh$.
\end{theorem}

\begin{proof}
The triangle inequality and the definitions of $\ERR$ and of $\Normtr{\cdot}_1$ entail that
\begin{equation*}
\Err{(u,p_F), (U, P_F)}
\leq
\Err{(u, p_F), (\widetilde{U}, \widetilde{P}_F)}
+
\Normtr{(U-\widetilde{U}, P_F- \widetilde{P}_F)}_1
\end{equation*}
for all $(\widetilde U, \widetilde P_F) \in \CR^\Dim \times \PolyPiec{1}$. Then, we derive the claimed error bound by combining Theorem~\ref{T:cont-infsup-discrete} with Proposition~\ref{P:consistency-error}.	
\end{proof}

In a sense, the error bound in Theorem~\ref{T:quasi-optimality} is not fully operative. In fact, the approximations of the total pressure and of the total fluid content are constrained by the relations \eqref{auxiliary-variables} in the definition \eqref{error-notion} of $\ERR$. This entails that the behavior of the best error in the right-hand side of \eqref{quasi-optimality} is not immediately clear. Interestingly, the inclusions \eqref{spaces-pT-m} and \eqref{spaces-u-pF} readily imply that the left-hand side of \eqref{quasi-optimality} is bounded from below as follows 
\begin{equation}
\label{operative-error-bound-lower}
\begin{split}
\Err{(u,p_F), (U, &P_F)} 
\gtrsim
\Big(   \mu^2\kappa \inf_{\widehat{U} \in \CR^\Dim}
\Norm{{u} - \widehat U}_\CRLabel^2 
+ \kappa \inf_{\widehat{P}_T \in \PolyPiecAvg{0}}
\NormLeb{ p_T - \widehat{P}_T}{\Domain}^2 + \\ 
&+\mu \inf_{\widehat{M} \in \PolyPiec{0}}
\Norm{m - \widehat{M}}_{\SobHD{\Domain}}^2
+ \mu \kappa^2 \inf_{\widehat{P}_F \in \PolyPiec{1}}
\Norm{{p}_F - \widehat P_F}_\DGLabel^2 
\Big) ^{\frac{1}{2}}.
\end{split}
\end{equation}
The right-hand side in this estimate is easier to analyze than the one of \eqref{quasi-optimality}, because each variable is approximated independently of the other ones. Thus, one may ask whether the above lower bound can be somehow reversed.   

In answering this question, our main device is the existence of an interpolant that is simultaneously near best in the $L^2$- and in the $H^{-1}$-norms. We define such an interpolant with the help of a variant of the operator $\Smt_\DGLabel$.

\begin{lemma}[First-order moment-preserving operator]
\label{L:smootherF}
There is a linear operator $\mathcal{F}:\PolyPiec{1} \to \SobH{\Domain}$ which fulfills the condition
\begin{equation}
\label{smootherF-moments}
\forall \MeshEl \in \Mesh, \; S \in \Poly{1}(\MeshEl)
\qquad
\int_{\MeshEl} S \mathcal{F}(Q) 
=
\int_\MeshEl S Q
\end{equation}
and enjoys the estimates
\begin{equation}
\label{smootherF-bounded}
\NormLeb{\Grad \mathcal{F}(Q)}{\Domain}
\lesssim 
\Norm{Q}_\DGLabel
\qquad \text{and} \qquad
\NormLeb{\mathcal{F}(Q)}{\Domain}
\lesssim
\NormLeb{Q}{\Domain}
\end{equation}
for all $Q \in \PolyPiec{1}$, where the hidden constants only depend on the shape parameter $\Shape{\Mesh}$ of $\Mesh$.
\end{lemma}

\begin{proof}
Recall the bubble functions $(S_\MeshEl)_{\MeshEl \in \Mesh}$ and the operator $\Smt: \PolyPiec{1} \to \SobH{\Domain}$ from \eqref{bubble-functions} and \eqref{smoother}, respectively. For all $Q \in \PolyPiec{1}$, we define
\begin{equation*}
\label{smootherF}
\mathcal{F}(Q) :=
\Smt(Q) + 
\sum_{\MeshEl \in \Mesh} \mathcal{F}_\MeshEl(Q-\Smt(Q))S_\MeshEl
\end{equation*}	
where, for $\MeshEl \in \Mesh$, the operator $\mathcal{F}_\MeshEl: \Leb{\Domain} \to \Poly{1}(\MeshEl)$ is uniquely determined through the problem
\begin{equation*}
\label{smootherF-local}
\forall S \in \Poly{1}(\MeshEl) 
\qquad
\int_\MeshEl S\mathcal{F}_\MeshEl(Q)S_\MeshEl
=
\int_\MeshEl SQ.
\end{equation*}
Note, in particular, that we indeed have $\mathcal{F}(Q) \in \SobH{\Domain}$. Since each bubble $S_{\MeshEl^\prime}$, $\MeshEl^\prime \in \Mesh$, vanishes outside $\MeshEl^\prime$, we have
\begin{equation*}
\int_\MeshEl S \mathcal{F}(Q)
=
\int_\MeshEl S \Smt(Q)
+
\int_\MeshEl S \mathcal{F}_\MeshEl(Q-\Smt(Q))S_\MeshEl
=
\int_{\MeshEl} SQ
\end{equation*} 
for all $\MeshEl \in \Mesh$ and $S \in \Poly{1}(\MeshEl)$. This confirms that $\mathcal{F}$ fulfills the condition \eqref{smootherF-moments}. The proof of the estimates in \eqref{smootherF-bounded} is similar to the one of \eqref{smoother-DG-bounded} in Proposition~\ref{P:smoother-bounded}, therefore we omit it.
\end{proof}

We are now in position to introduce the announced interpolant. Roughly speaking, it is defined as the adjoint of the operator $\mathcal{F}$ in the previous lemma.

\begin{lemma}[$L^2$- and $H^{-1}$-stable interpolant]
\label{L:interpolant}
Let $\Interp: \Leb{\Domain} \to \PolyPiec{0}$ be defined through the problem
\begin{equation}
\label{interpolant}
\forall Q \in \PolyPiec{0}
\qquad
\int_\Domain \Interp(q) Q
=
\int_\Domain q \mathcal{F}(Q)
\end{equation}
where $\mathcal{F}$ is as in Lemma~\ref{L:smootherF}. Then, we have
\begin{equation}
\label{interpolant-consistency}
\forall q \in \PolyPiec{1}
\qquad
\Interp(q) = \LebPro{\PolyPiec{0}}(q).
\end{equation}
Moreover, the following estimates hold true for all $q \in \Leb{\Domain}$
\begin{equation}
\label{interpolant-bounded}
\NormLeb{\Interp(q)}{\Domain}
\lesssim
\NormLeb{q}{\Domain}
\qquad \text{and} \qquad
\Norm{\Interp(q)}_{\SobHD{\Domain}}
\lesssim
\Norm{q}_{\SobHD{\Domain}}
\end{equation}
and the hidden constants only depend on the shape parameter $\Shape{\Mesh}$ of $\Mesh$.
\end{lemma}

\begin{proof}
Let $q \in \Leb{\Domain}$ be given. The second part of \eqref{smootherF-bounded} implies that
\begin{equation*}
\NormLeb{\Interp(q)}{\Domain}
=
\sup_{Q \in \PolyPiec{0}}
\dfrac{\int_\Domain \Interp(q) Q}{\NormLeb{Q}{\Domain}}
=
\sup_{Q \in \PolyPiec{0}}
\dfrac{\int_\Domain q \mathcal{F}(Q)}{\NormLeb{Q}{\Domain}}
\lesssim
\NormLeb{q}{\Domain}.
\end{equation*}
Similarly, Remark~\ref{R:sharpened-continuity-infsup} and the first part of \eqref{smootherF-bounded} reveal that
\begin{equation*}
\Norm{\Interp(q)}_{\SobHD{\Domain}}
\lesssim
\sup_{Q \in \PolyPiec{0}} \dfrac{\int_\Domain \Interp(q) Q}{\Norm{Q}_\DGLabel}
=
\sup_{Q \in \PolyPiec{0}} \dfrac{\int_\Domain q \mathcal{F}(Q)}{\Norm{Q}_\DGLabel}
\lesssim
\Norm{q}_{\SobHD{\Domain}}.
\end{equation*} 
This inequality and the previous one confirm that \eqref{interpolant-bounded} holds true. Finally, if $q \in \PolyPiec{1}$, we infer that
\begin{equation*}
\int_\Domain \Interp(q)Q
=
\int_\Domain q \mathcal{F}(Q)
=
\int_\Domain q Q
\end{equation*}
for all $Q \in \PolyPiec{0}$, as a consequence of \eqref{smootherF-moments}. Hence, we have $\Interp(q) = \LebPro{\PolyPiec{0}}(q)$.
\end{proof}

We are now in position to elaborate on the quasi-optimal error estimate in Theorem~\ref{T:quasi-optimality}. Roughly speaking, the boundedness of the interpolant $\Interp$ in the $L^2$- and in the $H^{-1}$-norms allows us to approximate both the total pressure and the total fluid content by $\Interp$. The condition \eqref{interpolant-consistency} serves to deal with the reduced integration, cf. Remark~\ref{R:reduced-integration}. Then, in a sense, we invert the relations in the second line of \eqref{auxiliary-variables} in a stable way, with the help of the equivalences stated in Proposition~\ref{P:cont-infsup-auxiliary-discrete}.

\begin{theorem}[Operative error bound]
\label{T:operative-error-bound}	
Assume that \eqref{penalty-parameter} holds true. Let the load term $(f, g) \in \SobHD{\Domain; \Dim} \times \SobHD{\Domain}$ be given and let $(u, p_F) \in \SobH{\Domain}^\Dim \times \SobH{\Domain}$ and $(U, P_F) \in \CR^\Dim \times \PolyPiec{1}$ be the corresponding solutions of the problems \eqref{Biot-weak} and \eqref{Biot-discrete}, respectively. Recall also the variables $p_T \in \LebH{\Domain}$ and $m \in \SobHD{\Domain}$ from \eqref{total-pressure} and \eqref{total-fluid-content}. Then, we have
\begin{equation*}
\label{operative-error-bound}
\begin{split}
&\Err{(u,p_F), (U,P_F)} 
\lesssim 
\Big(   \mu^2\kappa \inf_{\widehat{U} \in \CR^\Dim}
\Norm{u - \widehat U}_\CRLabel^2 
+ \kappa \inf_{\widehat P_T \in \PolyPiecAvg{0}}
\NormLeb{p_T - \widehat{P}_T}{\Domain}^2 +  \\
+&\mu \inf_{\widehat{M} \in \PolyPiec{0}}
\Norm{m - \widehat{M}}_{\SobHD{\Domain}}^2
+ \mu \alpha^2 \inf_{ \widehat{D} \in \PolyPiec{1} }
\NormLeb{\Div(u)-\widehat D}{\Domain}^2
+ \mu \kappa^2 \inf_{\widehat{P}_F \in \PolyPiec{1}}
\Norm{p_F - \widehat P_F}_\DGLabel^2 
\Big) ^{\frac{1}{2}}.
\end{split}
\end{equation*}
where $\ERR$ is as in \eqref{error-notion} and the hidden constant only depends on the shape parameter $\Shape{\Mesh}$ of $\Mesh$.
\end{theorem}

\begin{proof}
The first part of \eqref{cont-infsup-auxiliary-discrete} implies that there is a linear operator $\Rinv_\CRLabel: \PolyPiecAvg{0} \to \CR^\Dim$ such that
\begin{equation*}
\label{Rinv-CR}
\DivM(\Rinv_\CRLabel(Q_0)) = Q_0
\qquad \text{and} \qquad
\Norm{\Rinv_\CRLabel(Q_0)}_\CRLabel
\lesssim
\NormLeb{Q_0}{\Domain}
\end{equation*}
for all $Q_0 \in \PolyPiecAvg{0}$, i.e. $\Rinv_\CRLabel$ is a bounded right inverse of the broken divergence, cf. \cite[section~4.2.2]{Boffi.Brezzi.Fortin:13}. Similarly, the second part of \eqref{cont-infsup-auxiliary-discrete} reveals that there is a linear operator $\Rinv_\DGLabel: \PolyPiec{0} \mapsto \PolyPiec{1}$ such that
\begin{equation*}
\label{Rinv-DG}
\LebPro{\PolyPiec{0}}(\Rinv_\DGLabel(Q)) = Q
\qquad \text{and} \qquad
\Norm{\Rinv_\DGLabel(Q)}_\DGLabel
\lesssim
\NormLeb{Q}{\SobHD{\Domain}}
\end{equation*}
for all $Q \in \PolyPiec{0}$, i.e. $\Rinv_\DGLabel$ is a bounded right inverse of the $L^2$-orthogonal projection onto $\PolyPiec{0}$. Let $\widehat{U} \in \CR^\Dim$ and $\widehat{P}_F \in \PolyPiec{1}$ be given and define
\begin{equation}
\label{operative-error-bound-approximations}
\begin{aligned}
\widetilde{U} &:= \widehat{U} + \Rinv_\CRLabel(\LebPro{\PolyPiecAvg{0}}\Interp(\Div(u)) - \DivM(\widehat{U}))\\
\widetilde{P}_F &:= \widehat{P}_F + \Rinv_\DGLabel(\Interp(p_F) - \LebPro{\PolyPiec{0}}(\widehat{P}_F)).
\end{aligned}
\end{equation}
By definition, we have $\widetilde{U} \in \CR^\Dim$ and $\widetilde{P}_F \in \PolyPiec{1}$, as well as 
\begin{equation*}
\DivM(\widetilde{U}) = \LebPro{\PolyPiecAvg{0}}\Interp(\Div(u))
\qquad \text{and} \qquad
\LebPro{\PolyPiec{0}}(\widetilde{P}_F) = \Interp(p_F).
\end{equation*}
Thus, the auxiliary variables $\widetilde{P}_T \in \PolyPiecAvg{0}$ and $\widetilde{M} \in \PolyPiec{0}$ from \eqref{auxiliary-variables} are such that 
\begin{equation}
\label{operative-error-bound-auxiliary-variables}
\begin{aligned}
\widetilde{P}_T
&=
\lambda \LebPro{\PolyPiecAvg{0}}\Interp(\Div(u))
-
\alpha \LebPro{\PolyPiecAvg{0}}\Interp(p_F)
=
\LebPro{\PolyPiecAvg{0}}\Interp(p_T)\\
\widetilde{M}
&=
\alpha \LebPro{\PolyPiecAvg{0}}\Interp(\Div(u))
+
\sigma \Interp(p_F)
=
\Interp(m) - \alpha\LebPro{\R}\Interp(\Div(u))
\end{aligned}
\end{equation}
where $\LebPro{\R}$ denotes the $L^2$-orthogonal projection onto $\R$, i.e. onto the constant functions. Thus, by invoking Theorem~\ref{T:quasi-optimality} and recalling the definition of the error notion $\ERR$, we infer that
\begin{equation}
\label{operative-error-bound-proof}
\begin{split}
&\Err{(u,p_F), (U,P_F)} 
\lesssim 
\Big(   \mu^2\kappa 
\Norm{u - \widetilde U}_\CRLabel^2 
+ \kappa 
\NormLeb{p_T - \LebPro{\PolyPiecAvg{0}} \Interp(\widetilde{p}_T)}{\Domain}^2 +  \\
&+\mu 
\Norm{m - \Interp(m)}_{\SobHD{\Domain}}^2
+ \mu \alpha^2 
\Norm{\LebPro{\R}\Interp(\Div(u))}_{\SobHD{\Domain}}^2
+ \mu \kappa^2
\Norm{p_F - \widetilde P_F}_\DGLabel^2 
\Big) ^{\frac{1}{2}}.
\end{split}
\end{equation}
We estimate the five terms in the right-hand side one by one. The definition of $\widetilde{U}$, the boundedness of $\Rinv_\CRLabel$ and Lemma~\ref{L:interpolant} imply that
\begin{equation*}
\Norm{u-\widetilde{U}}_\CRLabel
\lesssim
\Norm{u-\widehat{U}}_\CRLabel 
+
\NormLeb{\Interp(\DivM(u-\widehat{U}))}{\Domain}
\lesssim
\Norm{u-\widehat{U}}_\CRLabel. 
\end{equation*}
According to the inclusion $p_T \in \LebH{\Domain}$, we have the identity $p_T - \LebPro{\PolyPiecAvg{0}} \Interp(p_T) = \LebPro{\LebH{\Domain}}(p_T - \Interp(p_T))$. Hence, it holds that
\begin{equation*}
\NormLeb{p_T - \LebPro{\PolyPiecAvg{0}} \Interp(\widetilde{p}_T)}{\Domain}
\leq
\NormLeb{p_T - \Interp(p_T)}{\Domain}
\lesssim
\inf_{\widehat{P}_T \in \PolyPiecAvg{0}}
\NormLeb{p_T - \widehat{P}_T}{\Domain}.
\end{equation*}
The second inequality holds true because $\Interp$ is a $L^2$-bounded projection onto $\PolyPiecAvg{0}$, in view of Lemma~\ref{L:interpolant}. Similarly, we have
\begin{equation*}
\Norm{m-\Interp(m)}_{\SobHD{\Domain}}
\lesssim
\inf_{\widehat{M} \in \PolyPiec{0}}
\Norm{m-\widehat{M}}_{\SobHD{\Domain}}
\end{equation*}
because $\Interp$ is a $H^{-1}$-bounded projection onto $\PolyPiec{0}$. Next, the second part of \eqref{cont-infsup-auxiliary-discrete}, the definition \eqref{interpolant} of $\Interp$, the inclusion $\Div(u) \in \LebH{\Domain}$ and Lemma~\ref{L:smootherF} yield
\begin{equation}
\label{operative-error-bound-proof2}
\begin{aligned}
\Norm{\LebPro{\R}\Interp(\Div(u))}_{\SobHD{\Domain}}
&\lesssim
\sup_{Q \in \PolyPiec{1}}
\dfrac{\int_\Domain \Div(u) (\mathcal{F}(\LebPro{\R}(Q))-\LebPro{\R}(Q)) }{\Norm{Q}_\DGLabel}\\
&=
\sup_{Q \in \PolyPiec{1}}
\dfrac{\int_\Domain (\Div(u) - \widehat{D}) (\mathcal{F}(\LebPro{\R}(Q))-\LebPro{\R}(Q)) }{\Norm{Q}_\DGLabel}
\end{aligned}
\end{equation}
for all $\widehat{D} \in \PolyPiec{1}$. Hence, we obtain
\begin{equation*}
\Norm{\LebPro{\R}\Interp(\Div(u))}_{\SobHD{\Domain}}
\lesssim
\inf_{ \widehat{D} \in \PolyPiec{1} } 
\NormLeb{\Div(u) - \widehat{D}}{\Domain}
\end{equation*}
according to Lemma~\ref{L:smootherF} and to the piecewise Poincar\`{e} inequality \cite[Theorem~5.3]{DiPietro.Ern:12}. Finally, the definition of $\widetilde{P}_F$, the boundedness of $\Rinv_\DGLabel$, Lemma~\ref{L:interpolant} and the piecewise Poincar\`{e} inequality imply that
\begin{equation*}
\Norm{p_F - \widetilde{P}_F}_\DGLabel 
\lesssim
\Norm{p_F - \widehat{P}_F}_\DGLabel + \Norm{\Interp(p_F - \widehat{P}_F)}_{\SobHD{\Domain}}
\lesssim
\Norm{p_F - \widehat{P}_F}_\DGLabel.
\end{equation*}
By inserting the above estimates into \eqref{operative-error-bound-proof}, we infer that
\begin{equation*}
\begin{split}
&\Err{(u,p_F), (U,P_F)} 
\lesssim 
\Big(   \mu^2\kappa 
\Norm{u - \widehat U}_\CRLabel^2 
+ \kappa \inf_{\widehat{P}_T \in \PolyPiecAvg{0}}
\NormLeb{p_T - \widehat{P}_T}{\Domain}^2 +  \\
&\quad+\mu \inf_{ \widehat{M} \in \PolyPiec{0} }
\Norm{m - \widehat M}_{\SobHD{\Domain}}^2
+ \mu \alpha^2 \inf_{ \widehat{D} \in \PolyPiec{1} }
\NormLeb{\Div(u)-\widehat D}{{\Domain}}^2
+ \mu \kappa^2
\Norm{p_F - \widehat P_F}_\DGLabel^2 
\Big) ^{\frac{1}{2}}.
\end{split}
\end{equation*}
We conclude by taking the infimum over $\widehat{U}$ and $\widehat{P}_F$.
\end{proof}

The error estimate in Theorem~\ref{T:operative-error-bound} does not exactly match the lower bound in \eqref{operative-error-bound-lower}, because of the additional summand
\begin{equation}
\label{operative-error-bound-additional}
\mu \alpha^2 \inf_{ \widehat{D} \in \PolyPiec{1} }
\NormLeb{\Div(u)-\widehat D}{{\Domain}}^2
\end{equation}
in the right-hand side. Two observations about this term are in order.

First, the presence of \eqref{operative-error-bound-additional} in our error estimate ultimately hinges on the Dirichlet boundary condition \eqref{Biot-BCs} on $u$. In fact, such condition is incorporated in the definition \eqref{CR-space} of the Crouzeix-Raviart space, entailing that the broken divergence $\DivM$ maps $\CR^\Dim$ onto $\PolyPiecAvg{0}$ (and not onto $\PolyPiec{0}$). Therefore, we must take $\LebPro{\PolyPiecAvg{0}}\Interp(\Div(u))$ (and not $\Interp(\Div(u))$) in the first line of \eqref{operative-error-bound-approximations}. This generates the term $\alpha\LebPro{\R}\Interp(\Div(u))$ in the second line of \eqref{operative-error-bound-auxiliary-variables}, which is then bounded by \eqref{operative-error-bound-additional}.

Second, for any $\Degree \geq 1$, we are allowed to replace \eqref{operative-error-bound-additional} with
\begin{equation*}
\mu \alpha^2 \inf_{ \widehat{D} \in \PolyPiec{\Degree} }
\NormLeb{\Div(u)-\widehat D}{{\Domain}}^2
\end{equation*}
in the error estimate of Theorem~\ref{T:operative-error-bound}, at the price of a possibly larger hidden constant. Indeed, we might construct the operator $\mathcal{F}$ in Lemma~\ref{L:smootherF} so that the condition \eqref{smootherF-moments} holds true for all polynomials of degree $\Degree$. Then, we may assume that $\widehat D \in \PolyPiec{\Degree}$ in \eqref{operative-error-bound-proof2}.

\section{Extensions of the main results}
\label{S:extensions}

In this final section we briefly outline some variants and generalizations of our previous results. We discuss an alternative discretization of the load terms, higher-order discretizations and more general boundary conditions than the ones in \eqref{Biot-BCs}.  

\subsection{Medius error analysis}
\label{SS:medius-analysis}

The use of the operators $\Smt_\CRLabel$ and $\Smt_\DGLabel$ in the right-hand side of the discretization \eqref{Biot-discrete} is necessary for the validity of the quasi-optimal error estimate in Theorem~\ref{T:quasi-optimality} and of the subsequent estimate in Theorem~\ref{T:operative-error-bound}, cf. \cite[section~4.2]{Veeser.Zanotti:18}. Still, this way of discretizing the right-hand side is not well-established in the framework of nonconforming finite element methods. More commonly, it is assumed that the loads $f$ and $g$ are more regular than in \eqref{loads-regularity}, for instance
\begin{equation*}
\label{loads-regularity-higher}
f \in \Leb{\Domain}^\Dim
\qquad \text{and} \qquad 
g \in \Leb{\Domain}.
\end{equation*}
When this is the case, the following discretization of the model problem can be considered
\begin{equation}
\label{Biot-discrete-medius-analysis}
\begin{gathered}[c]
\text{find} \quad (\underline U, \underline P_F)  \in \CR^\Dim \times \PolyPiec{1} \quad \text{such that}\\
\forall (V, Q_F) \in \CR^\Dim \times \PolyPiec{1} 
\quad \; 
B((\underline U, \underline P_F), (V, Q_F))
=
\int_\Domain f \cdot V + \int_\Domain g Q_F. 
\end{gathered}
\end{equation}

The results in the previous section readily allow one to derive an error estimate for this discretization, in the spirit of the medius error analysis of Gudi \cite{Gudi:10}, by proceeding as in \cite[Lemma~3.15]{Veeser.Zanotti:18b}. To see this, recall the solution $(U,P_F) \in \CR^\Dim \times \PolyPiec{1}$ of the discretization \eqref{Biot-discrete}. We first exploit the equivalence in Theorem~\ref{T:cont-infsup-discrete}. It follows that
\begin{equation*}
\begin{aligned}
\Normtr{(\underline U - U, \underline P_F - P_F)}_1
&\approx
\sup_{(V, Q_F) \in \CR^\Dim \times \PolyPiec{1}}
\dfrac{B((\underline U - U, \underline P_F - P_F), (V,Q_F))}{\Normtr{(V,Q_F)}_2}\\
& =
\sup_{(V, Q_F) \in \CR^\Dim \times \PolyPiec{1}}
\dfrac{\int_\Domain f \cdot (V-\Smt_\CRLabel(V)) + \int_\Domain g (Q_F - \Smt_\DGLabel(Q_F))}{\Normtr{(V,Q_F)}_2}
\end{aligned}
\end{equation*}
where the norms $\Normtr{\cdot}_1$ and $\Normtr{\cdot}_2$ are as in \eqref{norm-trial-discrete} and \eqref{norm-test-discrete}, respectively. According to Lemma~\ref{L:moments-preserved}, we have
\begin{equation*}
\int_{\partial \MeshEl} (V - \Smt_\CRLabel(V)) = 0
\qquad \text{and} \qquad
\int_\MeshEl (Q_F - \Smt_\DGLabel(Q_F)) = 0
\end{equation*}
for all $\MeshEl \in \Mesh$ and $(V, Q_F) \in \CR^\Dim \times \PolyPiec{1}$. Hence, we derive 
\begin{align*}
\int_\Domain f \cdot (V-\Smt_\CRLabel(V))
&\lesssim
\Big( \sum_{\MeshEl \in \Mesh} \mathrm{diam}(\MeshEl)^2 \NormLeb{f}{\MeshEl}^2 \Big)^\frac{1}{2}\Norm{V}_\CRLabel\\
\int_\Domain g (Q_F - \Smt_\DGLabel(Q_F))
&\lesssim
\Big( \sum_{\MeshEl \in \Mesh} \mathrm{diam}(\MeshEl)^2 \NormLeb{g}{\MeshEl}^2 \Big)^\frac{1}{2}\Norm{Q_F}_\DGLabel
\end{align*}
by invoking Poincar\`{e}-like inequalities and \eqref{Korn-piecewise}. We insert these estimates into the previous equivalence. Then, we apply the triangle inequality and we recall the definition \eqref{error-notion} of the error notion $\ERR$. We obtain
\begin{equation*}
\Err{(u,p_F), (\underline U, \underline P_F)}
\lesssim 
\Err{(u,p_F), (U,P_F)}
+
\Big( \sum_{\MeshEl \in \Mesh} \mathrm{diam}(\MeshEl)^2(\NormLeb{f}{\MeshEl}^2 + \NormLeb{g}{\MeshEl}^2) \Big)^{\frac{1}{2}}
\end{equation*}
where $(u,p_F) \in \SobH{\Domain}^\Dim \times \SobH{\Domain}$ is the solution of problem~\eqref{Biot-weak}. Finally, we establish an error bound for the discretization \eqref{Biot-discrete-medius-analysis} by invoking Theorem~\ref{T:operative-error-bound}. All the constants involved in our argument only depend on the shape parameter $\Shape{\Mesh}$ of $\Mesh$.

\subsection{Higher-order methods}
\label{SS:higher-order}

It is not difficult to design and analyze higher-order variants of the discretization \eqref{Biot-discrete} along the lines illustrated in section~\ref{S:discrete-problem}. To be more concrete, let $\Degree \geq 2$ be given. Perhaps, the most straight-forward extension of \eqref{spaces-pT-m} and \eqref{spaces-u-pF} consists in looking for approximations
\begin{align*}
U \in \CR_\Degree^\Dim \quad \text{of} \quad u
\qquad &\text{and} \qquad
P_F \in \PolyPiec{\Degree} \quad \text{of} \quad p_F\\
P_T \in \PolyPiecAvg{\Degree-1} \quad \text{of} \quad p_T
\qquad &\text{and} \qquad
M \in \PolyPiec{\Degree-1} \quad \text{of} \quad m.
\end{align*}
where the Crouzeix-Raviart space of degree $\Degree$ is defined as 
\begin{equation*}
\CR_\Degree := \{ S \in \PolyPiec{\Degree} \mid \forall \FaceEl \in \Faces, \: S_\FaceEl \in \Poly{\Degree-1}(\FaceEl) \;\; \int_\FaceEl \Jump{S}S_\FaceEl = 0 \}.
\end{equation*}

The definition of the space $\CR_\Degree$ and the results of \cite{Brenner:04} suggest to replace the form $A_\CRLabel$ and the norm $\Norm{\cdot}_\CRLabel$ from \eqref{form-CR} and \eqref{extended-norm-CR}, respectively, with
\begin{equation*}
A_\CRLabel(\widetilde U, V) := 
\int_\Domain \SymGradM(\widetilde{U}) \colon \SymGradM(V)
\qquad \text{and} \qquad
\Norm{\widetilde{u} - \widetilde U}_\CRLabel :=
\NormLeb{\SymGradM(\widetilde u - \widetilde U)}{\Domain}
\end{equation*}
because the jump penalization is no longer necessary for the validity of the piecewise Korn's inequality \eqref{Korn-piecewise}. In contrast, we do not modify the definitions of the form $A_\DGLabel$ and of the norm $\Norm{\cdot}_\DGLabel$ from \eqref{form-DG} and \eqref{extended-norm-DG}, respectively. Then, we define the form $B: (\CR_\Degree^\Dim \times \PolyPiec{\Degree}) \times (\CR_\Degree^\Dim \times \PolyPiec{\Degree}) \to \R$ as in \eqref{form-B}, with the only difference that the $L^2$-orthogonal projection $\LebPro{\PolyPiec{0}}(\widetilde{P}_F)$ is replaced with $\LebPro{\PolyPiec{\Degree-1}}(\widetilde{P}_F)$.

We consider the following discretization of the model problem~\eqref{Biot-weak}
\begin{equation}
\label{Biot-discrete-high-order}
\begin{gathered}[c]
\text{find} \quad (U, P_F)  \in \CR_\Degree^\Dim \times \PolyPiec{\Degree} \quad \text{such that}\\
\forall (V, Q_F) \in \CR_\Degree^\Dim \times \PolyPiec{\Degree} 
\quad \; 
B((U, P_F), (V, Q_F))
=
\left\langle f, \Smt_\CRLabel (V)\right\rangle + \left\langle g, \Smt_\DGLabel (Q_F)\right\rangle. 
\end{gathered}
\end{equation}
Here, the operator $\Smt_\CRLabel: \CR_\Degree^\Dim \to \SobH{\Domain}^\Dim$ should be defined so as to preserve the moments up to the order $\Degree-1$ on the interior faces of $\Mesh$ and up to the order $\Degree-2$ in the simplices of $\Mesh$, cf. \eqref{moments-preserved-CR}. Similarly, the operator $\Smt_\DGLabel: \PolyPiec{\Degree} \to \SobH{\Domain}$ should preserve the moments up to the order $\Degree-1$ on the interior faces and in the simplices of $\Mesh$, cf. \eqref{moments-preserved-DG}. Both the operators may be defined with the help of bubble functions, by a similar technique as in section~\ref{SS:FE-discretization}, cf. \cite[section~3.3]{Veeser.Zanotti:19b} and \cite[section~3.2]{Veeser.Zanotti:18b}. Alternatively, for sufficiently smooth loads, one might discretize the right-hand side as discussed in section~\ref{SS:medius-analysis}. 

The stability and the error analysis of the discretization \eqref{Biot-discrete-high-order} make use of the arguments in sections~\ref{SS:stability-discretization} and \ref{SS:error-analysis}. 

\begin{remark}[dG approximation of the displacement]
\label{R:dG-approximation-displacement}
The use of the space $\CR_\Degree^\Dim$ for the approximation of the displacement has two potential disadvantages. First, the construction of a nodal basis for this space is possibly involved, depending on $\Degree$ and $\Dim$. Second, a counterpart of the first equivalence in \eqref{cont-infsup-auxiliary-discrete} is known to hold only for certain combinations of $\Degree$ and $\Dim$, although we are not aware of any negative result. For these two reasons, one might consider a `fully dG' variant of the discretization~\eqref{Biot-discrete-high-order}, where the approximate displacement is sought in $\PolyPiec{\Degree}^\Dim$ and the differential operators acting on it are discretized as usual in discontinuous Galerkin methods. 
\end{remark}

\subsection{Mixed boundary conditions}
\label{SS:mixed-BCs}

Up to this point, we have assumed that homogeneous essential boundary conditions are enforced on $\partial \Domain$ for both the the displacement $u$ and the fluid pressure $p_F$, see \eqref{Biot-BCs}. A more general set of homogeneous boundary conditions is given by 
\begin{equation}
\label{Biot-BCs-mixed}
\begin{aligned}
u = 0 
\quad \text{on} \quad \Gamma_u 
\qquad &\text{and} \qquad 
(2 \mu \SymGrad(u) + p_T I)\Normal = 0 
\quad\text{on} \quad \Gamma_t \\
p_F = 0 
\quad \text{on} \quad \Gamma_p
\qquad &\text{and} \qquad 
k \Grad p_F \cdot \Normal = 0
\quad\text{on} \quad \Gamma_f
\end{aligned}
\end{equation}
where $\Gamma_u \cup \Gamma_t = \partial \Domain = \Gamma_p \cup \Gamma_f$ and $\Gamma_u \cap \Gamma_t = \emptyset = \Gamma_p \cap \Gamma_f$.

Denote by $|\cdot|$ the $(\Dim-1)$-dimensional Hausdorff measure on $\partial \Domain$. When $|\Gamma_t| > 0$ and/or $|\Gamma_f| > 0$, the spaces for the variables $u$ and $p_F$ and for the loads $f$ and $g$ are modified as usual when mixed boundary conditions are enforced. The only remarkable difference, compared to the discussion in section~\ref{S:continuous-problem}, is that the total pressure is given by 
\begin{equation*}
p_T = \lambda \Div(u) - \alpha p_F, \qquad p_T \in \Leb{\Domain} 
\end{equation*}
provided that $|\Gamma_t| > 0$. Moreover, the size of the total fluid content is measured in the norm of the space dual to the one to which $p_F$ belongs. Counterparts of the equivalences \eqref{cont-infsup-auxiliary} can be derived also in this case, entailing that the statements of Theorem~\ref{T:cont-infsup} and of Corollary~\ref{C:stability} still hold true, up to the necessary modifications.

Concerning the discretization of the problem~\eqref{Biot-weak} with the boundary conditions \eqref{Biot-BCs-mixed}, we proceed as in section~\ref{S:discrete-problem} with the following exceptions. We assume that each boundary face of the mesh $\Mesh$ is contained either in $\Gamma_u$ or in $\Gamma_t$ and either in $\Gamma_p$ or in $\Gamma_f$. The definition of the Crouzeix-Raviart space and of the forms $A_\CRLabel$ and $A_\DGLabel$ need to be modified as usual when mixed boundary conditions are involved. In particular, the jumps at the boundary should be penalized only on $\Gamma_u$ and on $\Gamma_p$. The approximate total pressure is given by 
\begin{equation*}
P_T = \lambda \DivM(U) - \alpha \LebPro{\PolyPiec{0}}(P_F),
\qquad P_T \in \PolyPiec{0}
\end{equation*} 
provided that $|\Gamma_t| > 0$. The operators $\Smt_\CRLabel$ and $\Smt_\DGLabel$ should preserve the averages not only on the interior faces of $\Mesh$ but also on the boundary faces that are contained in $\Gamma_t$ and in $\Gamma_f$, respectively, cf. Lemma~\ref{L:moments-preserved}. The stability and the error analyses of the resulting discretization proceed as indicated in sections~\ref{SS:stability-discretization} and \ref{SS:error-analysis}. For $|\Gamma_t| > 0$, the term \eqref{operative-error-bound-additional} may be omitted in the error estimate of Theorem~\ref{T:operative-error-bound}, as mentioned at the end of section~\ref{SS:error-analysis}.

\subsection*{Acknowledgment}
The authors wish to thank C. Kreuzer, A. Linke and A. Veeser for fruitful discussions concerning the approach described in section~\ref{SS:nonsymmetric-setting} and, in particular, the statement and the proof of Theorem~\ref{T:cont-infsup}.

\subsection*{Funding}
Pietro Zanotti was supported by the INdAM-GNCS through
the program ``Finanziamento giovani ricercatori 2019-2020'' and by the MIUR-PRIN project ``Numerical analysis of full and reduced order methods for partial differential equations''.


\begin{thebibliography}{10}

\bibitem{Arnold:93}
{\sc D.~N. Arnold}, {\em On nonconforming linear-constant elements for some variants of the {S}tokes equations}, Istit.
Lombardo Accad. Sci. Lett. Rend. A, 127 (1993).

\bibitem{Bartels.Wang:20}
{\sc S.~Bartels and Z.~Wang}, {\em Orthogonality relations of {C}rouzeix-{R}aviart and {R}aviart-{T}homas finite element spaces}, arXiv:2005.02741,  (2020).

\bibitem{Berger.Bordas.Kay.Tavener:15}
{\sc L.~Berger, R.~Bordas, D.~Kay, and S.~Tavener}, {\em Stabilized lowest-order finite element approximation for linear three-field poroelasticity}, SIAM J. Sci. Comput., 37 (2015), pp.~A2222--A2245.

\bibitem{Boffi.Brezzi.Fortin:13}
{\sc D.~Boffi, F.~Brezzi, and M.~Fortin}, {\em Mixed {F}inite {E}lement {M}ethods and {A}pplications}, vol.~44 of Springer Series in Computational Mathematics, Springer, Heidelberg, 2013.

\bibitem{Brenner:04}
{\sc S.~C. Brenner}, {\em Korn's inequalities
	for piecewise {$H^1$} vector fields}, Math. Comp., 73 (2004), pp.~1067--1087.

\bibitem{Carstensen.Schedensack:15}
{\sc C.~Carstensen and M.~Schedensack}, {\em Medius analysis and comparison results for first-order finite element methods in linear elasticity}, IMA J. Numer. Anal., 35 (2015), pp.~1591--1621.

\bibitem{Chen.Luo.Feng:13}
{\sc Y.~Chen, Y.~Luo, and M.~Feng}, {\em Analysis of a discontinuous {G}alerkin method for the {B}iot's consolidation problem}, Appl. Math. Comput., 219
(2013), pp.~9043--9056.

\bibitem{DiPietro.Ern:12}
{\sc D.~A. Di~Pietro and A.~Ern}, {\em {M}athematical {A}spects of {D}iscontinuous {G}alerkin {M}ethods}, vol.~69 of Math\'ematiques \& Applications (Berlin) [Mathematics \& Applications], Springer, Heidelberg, 2012.

\bibitem{Ern.Guermond:04}
{\sc A.~Ern and J.-L. Guermond}, {\em Theory and {P}ractice of {F}inite
	{E}lements}, vol.~159 of Applied Mathematical Sciences, Springer-Verlag, New
York, 2004.

\bibitem{Gudi:10}
{\sc T.~Gudi}, {\em A new error analysis for discontinuous finite element methods for linear elliptic problems}, Math. Comp., 79 (2010), pp.~2169--2189.

\bibitem{Haga.Osnes.Langtangen:12}
{\sc J.~B. Haga, H.~Osnes, and H.~P. Langtangen}, {\em On the causes of pressure oscillations in low-permeable and low-compressible porous media},
International Journal for Numerical and Analytical Methods in Geomechanics, 36 (2012), pp.~1507--1522.

\bibitem{Hong.Kraus:18}
{\sc Q.~Hong and J.~Kraus}, {\em Parameter-robust stability of classical three-field formulation of {B}iot's consolidation model}, Electron. Trans. Numer. Anal., 48 (2018), pp.~202--226.

\bibitem{Hu.Rodrigo.Gaspar.Zikatanov:17}
{\sc X.~Hu, C.~Rodrigo, F.~J. Gaspar, and L.~T. Zikatanov}, {\em A nonconforming finite element method for the {B}iot's consolidation model in poroelasticity}, J. Comput. Appl. Math., 310 (2017), pp.~143--154.

\bibitem{Khan.Zanotti}
{\sc A.~Khan and P.~Zanotti}, {\em A nonsymmetric approach and a quasi-optimal and robust discretization for the Biot's consolidation model. Part {II} -- Numerical aspects}, in preparation.

\bibitem{Korsawe.Starke:05}
{\sc J.~Korsawe and G.~Starke}, {\em A least-squares mixed finite element method for {B}iot's consolidation problem in porous media}, SIAM J. Numer. Anal., 43 (2005), pp.~318--339.

\bibitem{Kumar.Oyarzua.RuizBaier.Sandilya:20}
{\sc S.~Kumar, R.~Oyarz\'{u}a, R.~Ruiz-Baier, and R.~Sandilya}, {\em Conservative discontinuous finite volume and mixed schemes for a new four-field formulation in poroelasticity}, ESAIM Math. Model. Numer. Anal.,
54 (2020), pp.~273--299.

\bibitem{Lee:16}
{\sc J.~J. Lee}, {\em Robust error analysis of coupled mixed methods for {B}iot's consolidation model}, J. Sci. Comput., 69 (2016), pp.~610--632.

\bibitem{Lee.Mardal.Winther:17}
{\sc J.~J. Lee, K.-A. Mardal, and R.~Winther}, {\em Parameter-robust
	discretization and preconditioning of {B}iot's consolidation model}, SIAM J.
Sci. Comput., 39 (2017), pp.~A1--A24.

\bibitem{Mardal.Rognes.Thompson:20}
{\sc K.-A. Mardal, M.~E. Rognes, and T.~B. Thompson}, {\em Accurate discretization of poroelasticity without {D}arcy stability -- {S}tokes-{B}iot stability revisited}, arXiv:2007.10012,  (2020).

\bibitem{Nordbotten:16}
{\sc J.~M. Nordbotten}, {\em Stable cell-centered finite volume discretization for {B}iot equations}, SIAM J. Numer. Anal., 54 (2016), pp.~942--968.

\bibitem{Oyarzua.RuizBaier:16}
{\sc R.~Oyarz\'{u}a and R.~Ruiz-Baier}, {\em Locking-free finite element
	methods for poroelasticity}, SIAM J. Numer. Anal., 54 (2016), pp.~2951--2973.

\bibitem{Phillips.Wheeler:07}
{\sc P.~J. Phillips and M.~F. Wheeler}, {\em A coupling of mixed and continuous
	{G}alerkin finite element methods for poroelasticity. {II}. {T}he
	discrete-in-time case}, Comput. Geosci., 11 (2007), pp.~145--158.

\bibitem{Phillips.Wheeler:08}
\leavevmode\vrule height 2pt depth -1.6pt width 23pt, {\em A coupling of mixed
	and discontinuous {G}alerkin finite-element methods for poroelasticity},
Comput. Geosci., 12 (2008), pp.~417--435.

\bibitem{Riedlbeck.DiPietro.Ern.Granet.Kazymyrenko:17}
{\sc R.~Riedlbeck, D.~A. Di~Pietro, A.~Ern, S.~Granet, and K.~Kazymyrenko},
{\em Stress and flux reconstruction in {B}iot's poro-elasticity problem with
	application to a posteriori error analysis}, Comput. Math. Appl., 73 (2017),
pp.~1593--1610.

\bibitem{Rodrigo.Gaspar.Hu.Zikatanov:16}
{\sc C.~Rodrigo, F.~J. Gaspar, X.~Hu, and L.~T. Zikatanov}, {\em Stability and
	monotonicity for some discretizations of the {B}iot's consolidation model},
Comput. Methods Appl. Mech. Engrg., 298 (2016), pp.~183--204.

\bibitem{Rodrigo.Hu.Ohm.Adler.Gaspar.Zikatanov:18}
{\sc C.~Rodrigo, X.~Hu, P.~Ohm, J.~H. Adler, F.~J. Gaspar, and L.~T.
	Zikatanov}, {\em New stabilized discretizations for poroelasticity and the
	{S}tokes' equations}, Comput. Methods Appl. Mech. Engrg., 341 (2018),
pp.~467--484.

\bibitem{Veeser.Zanotti:18}
{\sc A.~Veeser and P.~Zanotti}, {\em Quasi-optimal nonconforming methods for
	symmetric elliptic problems. {I}---{A}bstract theory}, SIAM J. Numer. Anal.,
56 (2018), pp.~1621--1642.

\bibitem{Veeser.Zanotti:19b}
{\sc A.~Veeser and P.~Zanotti}, {\em Quasi-optimal
	nonconforming methods for symmetric elliptic problems.
	{II}---{O}verconsistency and classical nonconforming elements}, SIAM J.
Numer. Anal., 57 (2019), pp.~266--292.

\bibitem{Veeser.Zanotti:18b}
{\sc A.~Veeser and P.~Zanotti}, {\em Quasi-optimal nonconforming methods for
	symmetric elliptic problems. {III}---{D}iscontinuous {G}alerkin and other
	interior penalty methods}, SIAM J. Numer. Anal., 56 (2018), pp.~2871--2894.

\bibitem{Verfuerth:13}
{\sc R.~Verf\"{u}rth}, {\em A {P}osteriori {E}rror {E}stimation {T}echniques
	for {F}inite {E}lement {M}ethods}, Numerical Mathematics and Scientific
Computation, Oxford University Press, Oxford, 2013.

\bibitem{Yi.13}
{\sc S.-Y. Yi}, {\em A coupling of nonconforming and mixed finite element
	methods for {B}iot's consolidation model}, Numer. Methods Partial
Differential Equations, 29 (2013), pp.~1749--1777.

\bibitem{Yi.14}
\leavevmode\vrule height 2pt depth -1.6pt width 23pt, {\em Convergence analysis
	of a new mixed finite element method for {B}iot's consolidation model},
Numer. Methods Partial Differential Equations, 30 (2014), pp.~1189--1210.

\end{thebibliography}
\end{document}